\documentclass[twoside,amsmath,leqno]{siamltex}

\usepackage{amsmath,amsxtra,amssymb,color,
latexsym,epsfig,
subfigure,fancybox,epsfig}
\usepackage{enumerate}
\usepackage{enumitem}
\usepackage{epsfig}
\usepackage{cases}
\usepackage{hyperref}

\def\para#1{\vskip .2\baselineskip\noindent{\bf #1}}

\newtheorem{thm}{Theorem}[section]
\newtheorem {asp}{Assumption}[section]
\newtheorem{lm}{Lemma}[section]
\newtheorem{prop}{Proposition}[section]
\newtheorem{deff}{Definition}[section]

\newtheorem{cor}[thm]{Corollary}

\newtheorem{rem}{Remark}[section]

\numberwithin{equation}{section}

\DeclareMathOperator{\ess}{ess}

\newcommand{\eps}{\varepsilon}

\newcommand{\G}{\mathcal{G}}
\newcommand{\C}{\mathcal{C}}
\newcommand{\A}{\mathcal{A}}
\newcommand{\J}{\mathcal{J}}
\newcommand{\M}{\mathcal{M}}
\newcommand{\F}{\mathcal{F}}

\newcommand{\E}{\mathbb{E}}

\newcommand{\N}{\mathbb{N}}
\newcommand{\I}{\mathbb{I}}
\newcommand{\JJ}{\mathbb{J}}
\newcommand{\PP}{\mathbb{P}}

\newcommand{\R}{\mathbb{R}}

\newcommand{\abs}[1]{\left\vert#1\right\vert}

\numberwithin{equation}{section}
\newcommand{\wdt}{\widetilde}
\newcommand{\di}{\text{div}}
\newcommand{\tr}{\text{tr}}

\newcommand{\bed}{\begin{displaymath}}
\newcommand{\eed}{\end{displaymath}}
\newcommand{\bea}{\bed\begin{array}{rl}}
\newcommand{\eea}{\end{array}\eed}
\newcommand{\ad}{&\!\!\disp}
\newcommand{\aad}{&\disp}
\newcommand{\barray}{\begin{array}{ll}}
\newcommand{\earray}{\end{array}}
\def\disp{\displaystyle}

\newcommand{\1}{\boldsymbol{1}}
\newcommand{\beq}{\begin{equation}}
\newcommand{\eeq}{\end{equation}}

\def\bar{\overline}
\def\hat{\widehat}
\def\a.s{\text{\;a.s.\;}}

\begin{document}

\title{Second-Order Fast-Slow Stochastic Systems
}

\author{Nhu N. Nguyen\thanks{Department of Mathematics and Applied Mathematical Sciences, University of Rhode Island, Kingston, RI 02881, nhu.nguyen@uri.edu. The research of this author was supported in part by the AMS-Simons travel grant.}
%\orcid{0000-0001-5772-9730}
\and George Yin\thanks{Department of Mathematics,
University of Connecticut, Storrs, CT 06269,
gyin@uconn.edu. The research of this author was supported in part by the National Science Foundation under grant DMS-2204240.}}
%\orcid{0000-0002-2951-0704}

\maketitle

\begin{abstract}
This paper focuses on systems of nonlinear second-order stochastic differential equations with multi-scales. The motivation for our study stems from mathematical physics and statistical mechanics, for examples, Langevin dynamics and
stochastic acceleration in a random environment.
Our  effort  is to carry out  asymptotic analysis
to establish large deviations principles. Our focus is on obtaining the desired results for systems under weaker conditions. When the fast-varying process is a diffusion, neither Lipschitz continuity nor linear growth needs to be assumed.
	Our approach is based on combinations of
the intuition from Smoluchowski-Kramers approximation,
and the methods initiated in \cite{Puh16} relying on the concepts of relatively large deviations compactness and the identification of rate functions.
 When the fast-varying process is under a general setup with
no specified structure,  the paper establishes the large deviations principle of the underlying system under the assumption on the local large deviations principles of the corresponding first-order system.
\end{abstract}



\pagestyle{myheadings} \thispagestyle{plain}
\markboth{NGUYEN, YIN}{Second-Order  Stochastic Systems}

\section{Introduction}\label{sec:int}
In recent years, much effort has been devoted to analyzing stochastic systems  arising from a wide variety of fields. For example, averaging principle for complex Ginzburg--Landau equations was studied by  Gao \cite{Gao}, homogenization in ergodic media was treated in Chen et al. \cite{CCKW},  homogenization of stochastic convection-diffusion equation was studied in Bessaih et al.
\cite{BEM22}, mean field limits of particle-based stochastic systems were obtained in Isaacson et al. \cite{IMS},
Freidlin--Wentzell type large deviation results were obtained for multi-scale stochastic partial differential equations in Hong et al. \cite{HLL}.
One of the salient features in many applications is time scale separation. For example,
in Khasminskii and Yin \cite{KhY}, we treated diffusions with fast and slow motions.
 Although the first-order stochastic differential equations have been analyzed extensively,  properties associated with the  second-order stochastic differential equations are less well known.  In applications,  for example,
 in numerous systems in mathematical physics and statistical mechanics, such equations naturally arise; see for example, the work of Kesten and Papanicolaou in \cite{KP79,KP80}.

Because of the need, this paper is devoted to fully nonlinear second-order stochastic systems.
We begin with the study of a class of
second-order stochastic differential equations
\begin{equation}\label{eq:F-setup}
\begin{cases}
\eps^2\ddot{X}^\eps_t=F^\eps_t(X^\eps_t,Y^\eps_t)-\lambda^\eps_t(X^\eps_t,Y^\eps_t)\dot{X}^\eps_t,\quad X^\eps_0=x_0^\eps\in\R^d,\quad\dot{X}^\eps_0=x_1^\eps\in \R^d,\\
\dot Y^\eps_t=\dfrac 1\eps  b^\eps_t(X^\eps_t,Y^\eps_t)+\dfrac 1{\sqrt\eps}\sigma^\eps_t(X^\eps_t,Y^\eps_t)\dot W_t,\quad Y^\eps_0=y_0^\eps\in\R^l,
\end{cases}
\end{equation}
where $\eps>0$ is a small parameter.
Equation \eqref{eq:F-setup} is a multi-scale and fully nonlinear system.
In the above, for each $\eps>0$, $F^\eps_t(x,y):\R_+\times\R^d\times\R^l\to\R^d$,
$\lambda^\eps_t(x,y):\R_+\times\R^d\times\R^l\to\R$, $b^\eps_t(x,y):\R_+\times\R^d\times\R^l\to\R^l$, $\sigma^\eps_t(x,y):\R_+\times\R^d\times\R^l\to\R^{l\times m}$ are measurable functions of their arguments $(t,x,y)$, and $W_t$ is an $m$-dimensional
vector-valued standard
Brownian motion with $\dot W_t$ being its formal derivative.
For each $\eps>0$, the weak solution of \eqref{eq:F-setup}
is defined
in the usual way, i.e., there exists a suitable
probability space
and an adapted Brownian motion $W_t$ such that there are adapted processes $(X^\eps,Y^\eps)$ satisfying the  system of stochastic integral equations
corresponding to \eqref{eq:F-setup} almost surely.

Equations given in \eqref{eq:F-setup} may be considered as a singular perturbation problem with multiple-time scales.
Intuitively, as $\eps^2\to 0$ in the first equation of \eqref{eq:F-setup}, it can be approximated by a first-order equation, whereas
 $Y^\eps$ in the second equation of  \eqref{eq:F-setup} can be viewed as a fast-varying process, which will be so referred to in what follows.
 Further heuristic reasoning can be found in the beginning of Section \ref{sec:num}.

Our main effort is devoted to obtaining asymptotic properties of the underlying systems.
Under mild conditions,
we establish the large deviations principle (LDP for short) for the family of coupled processes $\{(X^\eps,\mu^\eps)\}_{\eps>0}$ with
  $\mu^\eps$ being the occupation measures
of the fast-varying process $Y^\eps$.
Neither Lipschitz continuity nor growth condition of $F^\eps,b^\eps,\sigma^\eps$ is assumed. From the LDP of such couples, we can obtain  averaging and
large deviations principles
for $\{X^\eps\}_{\eps>0}$.
In addition,
continuing our investigation, in this paper, we further reveal  asymptotic properties without assuming specific structure of
the fast process.
In lieu of
\eqref{eq:F-setup}, we consider
\begin{equation}\label{eq:setup}
\eps^2\ddot X^\eps_t=F^\eps_t(X^\eps_t,\xi^\eps_{t})-\lambda^\eps_t(X^\eps_t,\xi^\eps_{t})\dot X^\eps_t,
\end{equation}
with $\xi^\eps_{t}$ being a   process
without  a specified structure. We refer to  $\xi^\eps_{t}$ as a  fast-varying process for similar reason as that of $Y^\eps_t$ in the previous paragraph.

Why do we care of the second-order stochastic systems? This is because numerous problems in physics, statistical mechanics,  and engineering, etc., involve such systems.
In fact, in the study of ordinary differential equations, we encountered many second-order equations, including Airy's equations, Duffing equations, Li\'enard equations,
Rayleigh's equations,
etc. They have been used in a wide variety of applications.
Adding stochastic perturbations to these equations leads to second-order stochastic differential equations of various kind.

To further illustrate,
consider the motions of a net of particles in a net of random force fields,  described by the Newton's law as
$
\ddot x_\eps(t)=\wdt F_\eps(t,\omega,x_\eps(t),\dot x_\eps(t),\chi_\eps(t)),
$
where $x_\eps(t)$ denotes the location of the particles at time $t$. The
 $\wdt F_\eps$ denotes the random force fields depending on time $t$,
sample point $\omega$,
the particle's locations $x_\eps$, the particle's velocities $\dot x_\eps$, and   the random environments $\chi_\eps(t)$ interacting with the system.
To begin, turbulent diffusions and stochastic accelerations were
considered by Kesten and Papanicolaou in \cite{KP79,KP80}
under
suitable conditions.
Here we focus on the motions of
particles, in which the Reynolds number (see e.g., \cite{Pur77} for a definition) is very small so that inertial effects are negligible
compared to the damping force
by assuming that
$$
\wdt F_\eps(t,\omega,x_\eps(t),\dot x_\eps(t),\chi_\eps(t))=F_\eps(t,x_\eps(t),\chi_\eps(t))-
\frac{\lambda_\eps(t,x_\eps(t),\chi_\eps(t))}
{\eps}\dot x_\eps(t).
$$
Now, by scaling $X^\eps_t:=x_\eps(t/\eps)$,
and $\xi_{t}^\eps:=\chi_\eps(t/\eps)$,
the system can be rewritten as \eqref{eq:setup}.
One of the examples of
$\chi_\eps(t)$ is
a diffusion process. In this case,
 $\xi_{t}^\eps$ is a fast diffusion process
that is
fully coupled with the system, which leads to the system of equations in \eqref{eq:F-setup}.
Another motivation is from the averaging and large deviations principles for systems of stochastic differential equations.
System \eqref{eq:F-setup} can be viewed as the second-order version of the problem considered in \cite{Lip96} and references therein. It should be mentioned that
there have been much recent interests in studying stochastic second-order systems in random environment. For example, the work \cite{XZW21}
studied stochastic Hamiltonian systems living in random environments with the random environment represented by a random switching process.

Because $\eps$ is a small parameter, as $\eps$ is getting smaller and smaller, we expect the system to display certain limit behavior,  in which the
averaging principle plays an important role in studying heterogeneity
that  often occurs in physics as well as in biology, economics, queuing theory, game theory,
among others; see, e.g., \cite{FGS12}.
Typically, analyzing and simulating heterogeneous models are much more challenging than the corresponding homogeneous models, in which the heterogeneous property is replaced by its average value. The averaging principle for a system guarantees the validity of this replacement.
On the other hand, the
LDPs (see \cite{DS89,DZ98}), characterizing
quantitatively the rare events, play an important role in many areas
with a wide range of applications.
To mention just a few,
they include
equilibrium and non-equilibrium statistical mechanics,
multi-fractals,
thermodynamics
of chaotic systems,
among others \cite{Tou09}.
By establishing the LDPs
for system \eqref{eq:F-setup} and \eqref{eq:setup}, we provide an insight about the motions of (small) particles in random force fields, which is heterogeneous and the heterogeneity is allowed to interact with the system.  Not only will it illustrate
averaging of the heterogeneity works in this case, but also provide the picture of the dynamics
around the averaged system.

 From the
development
of homogenization
and large deviations
point of view,
 much effort
  has been devoted to studying
 averaging and large deviations principles
 of the first-order differential equations under random environment (given by diffusion process, switching process, wideband noise, and others) in the setting of fast-slow systems.
Such problems have been addressed in \cite{Gui03,Kif09,Kus20,Xu22,XW23,Ver99,Ver00} under certain settings, in which, the fast process is often not fully coupled with the slow system.
Very recently, the question for the fully-coupled system  was addressed
in \cite{Puh16}.
Some other related studies
can be found in \cite{BDG18, Yinhe,Lip96}.
Reference  \cite{Kus20} considered systems under wideband noise; \cite{PIX21} studied systems under rough path noise;
\cite{CG20,CX22,CC18} investigated
systems in infinite dimensional settings.
 In contrast to
 the systems considered in the aforementioned works with emphases on first-order equations,
  we consider systems of second-order differential equations of the forms
  \eqref{eq:F-setup} and \eqref{eq:setup}.
 From a statistical physics
 point of view, there were some works treating the stochastic accelerations and the Langevin equations
 such as
 \cite{CF05,Fre04,XY22} for the study of Smoluchowski-Kramers approximation,
the work \cite{Freidlin} for the LDPs,  \cite{Cheng} for the MDPs (moderate deviations principles) in the
absence of the random environment,
and \cite{NY-JMP,NY-JMP2} for the LDPs of Langevin systems with random environment under certain specific settings.
To the best of our knowledge,
this paper is one of the first works addresses the problem of
averaging and
large deviations principle
for   second-order equations in random environment
that
are fully coupled.
We  establish the
LDPs under mild and natural conditions.

To establish the desired LDPs
for system \eqref{eq:F-setup}, in light of the work of
\cite{Lip96,Puh16} on the first-order SDEs,
we first establish the LDP for the family of coupled processes $\{(X^\eps,\mu^\eps)\}_{\eps>0}$, where $\mu^\eps$ is defined as a  random occupation measure of $Y^\eps$.
Then, the LDPs for the families
$\{X^\eps\}_{\eps>0}$ and/or $\{\mu^\eps\}_{\eps>0}$ can be
handled
 by some standard projection techniques in the large deviations theory.
Without
 assuming any regularity of $F^\eps$, $b^\eps$, and $\sigma^\eps$,
we could not establish a ``good" connection between the solution of the second-order equations
and
the corresponding first-order equations.
Our approach is based on a combination of
the approach of Puhaskii in \cite{Puh16} for the first-order coupled system (namely, obtaining the relatively large deviations compactness and then carefully identifying the rate functions), and the intuition of Smoluchowski-Kramers approximation.
	To establish LDPs of SDEs, the weak convergence methods
initiated by Dupuis and Ellis \cite{DE11},
have been used by many authors (see e.g., \cite{BDG18,Freidlin,DE11,Kus20} and references therein), which is shown to be effective
to prove the LDPs for
many systems.
However,  using
weak convergence approach for our problem
may require stronger assumptions such as Lipschitz continuity of coefficients in equation of $X^\eps$ (as shown in, e.g., \cite[Assumption 2.1.]{BDG18} and \cite[Hypothesis 1]{Freidlin}).
Such conditions are
 needed in Budhiraja,
Dupuis, and
 Ganguly
\cite{BDG18} because of the need to
 prove the lower bound \cite[(2.13)]{BDG18}, in which some uniqueness properties of auxiliary optimal controls are required.
The paper \cite{BDG18} studied the first-order SDEs with a fast-varying jump process, the aforementioned difficulty
arises in
 \cite{BDG18} due to the presence of multiple time scales rather than the presence of the jump process.
Here, we are dealing with fully nonlinear second-order stochastic systems with multi-scales, but we do not use the weak convergence method  to avoid requiring the Lipschitz continuity and other growth conditions.
In
\cite{Freidlin},  Cerrai and Freidlin
considered the second-order SDEs without coupling with another fast-varying processes. To establish the desired convergence, Lipschitz continuities of coefficients in the system are necessary.
In \cite{FeKu}, Feng and Kurtz  introduced the HJB equations/viscosity solutions approach. In \cite[Section 11.6]{FeKu}, first-order SDEs is considered, and conditions for the validity of LDPs are derived. However, these conditions rely on the existence of functions possessing certain  properties,
 which
 are often difficult to verify
in terms of the coefficients.
Although Feng and Kurtz were able
to provide explicit conditions on the coefficients, a key requirement is that $\sigma^\eps_t(x, y)$
being independent of
 $x$ (see \cite[Lemma 11.60 on p.278]{FeKu}).
As to be seen later, we
do not need the
Lipschitz continuity
for \eqref{eq:F-setup} neither do we need
$\sigma^\eps_t(x, y)$ being independent of $x$ as in \cite{FeKu}. In this paper,
we manage to
 establish LDPs of $X^\eps$ in multiscale and fully coupled system \eqref{eq:F-setup} under mild conditions, which is
 another of our goal.

To establish the desired LDP for
the system under general fast random process \eqref{eq:setup},
we have to use a different approach.
We assume that the corresponding first-order equation satisfies the local LDP, which will be shown to be verifiable and  satisfied in many problems.
To prove the LDP, we show that the family of $\{X^\eps\}_{\eps>0}$ is exponentially tight and satisfies the local LDP.

The rest of the paper is arranged as follows. We divide the presentation of the rest of the paper into two parts. The first part, Section \ref{sec:for},
 is devoted to the second-order systems with a fast-varying diffusion \eqref{eq:F-setup}.
  Section \ref{sec:re} formulates the problem and states
 the results.
 The detailed proof of results is  provided in Section \ref{sec:prof}.
 The second part of the paper,  presented in
Section \ref{sec:main2}, substantially extends the results to that of
second-order equations with general fast-varying random processes \eqref{eq:setup}.
The formulation, conditions, results, and detailed proofs are presented.
	Finally, Section \ref{sec:example} presents two examples
to illustrate our formulation and results.

\section{Fast-Slow Second-Order Systems with Fast Diffusion}
\label{sec:for}

\subsection{Notation, Formulation, and Results}\label{sec:re}

Throughout the paper, $|\cdot|$ denotes an Euclid norm
while $\|\cdot\|$ indicates the operator sup-norm,
$\C(\mathcal X,\mathcal Y)$ is
the space of continuous functions from $\mathcal X$ to $\mathcal Y$ and if $\mathcal Y$ is an Euclid space, we write $\C(\mathcal X,\mathcal Y)$ as $\C(\mathcal X)$ for simplicity.
Let $\M(\R^l)$
% [resp., M1(Rd)]
be the set of finite
measures on $\R^l$ endowed with the weak topology, and
$\mathcal P(\R^l)$ be the set of probability densities $m(y)$ on $\R^l$ such that $m \in\mathbb W^{1,1}_{\rm loc}(\R^d)$ and
$\sqrt m\in\mathbb W^{1,2}(\R^l)$, where $\mathbb W^{1,2}(\R^l)$ (resp., $\mathbb W^{1,1}_{\rm loc}(\R^d)$) is the Sobolev space (resp., local Sobolev space) with
suitable
exponents, and
$\C_0^1(\R^l)$ be the space of continuously differentiable functions with compact supports in $\R^l$.
Let $\C_{\uparrow}(\R_+,\M(\R^l))$ represent the subset of $\C(\R_+,\M(\R^l))$ of functions $\mu =
(\mu_t,t \in\R_+)$ such that $\mu_t -\mu_s$ is an element of $\M(\R^l)$ for $t\geq s$ and $\mu_t(\R^l) = t$. It is endowed with the subspace topology and is a complete separable metric space,
being closed in $\C(\R_+,\M(\R^l))$.
We define the random
process $\mu^\eps = (\mu^\eps_t, t \in\R_+)$ of the fast process $Y^\eps$ by
\begin{equation}\label{eq-mu} \barray
\mu^\eps_t(\A) := \int_0^t \1_{\A}(Y^\eps_s)ds,\quad \forall\A\in\mathcal B(\R^l).\earray
\end{equation}
Then, $\mu^\eps$ is a random element of $\C_{\uparrow}(\R_+,\M(\R^l))$
and we can regard $(X^\eps,\mu^\eps)$
as a random element of $\C(\R_+,\R^d)\times\C_{\uparrow}(\R_+,\M(\R^l))$.
%It is also noting
Note that the
elements of $\C_{\uparrow}(\R_+,\M(\R^l))$ can be regarded as a $\sigma$-finite measures on $\R_+\times\R^l$. As a result,
% so that
we use the notation $\mu(dt, dy)$ for $\mu\in\C_{\uparrow}(\R_+,\M(\R^l))$.
For a symmetric positive
definite  matrix $A$ and matrix $z$ of suitable dimensions,
%which may be a vector,
we define $\|z\|_{A} := z^\top A z$. Following Puhaski's notation,
%It is noted that
$\|z\|_{A}$ can be either matrices or numbers, depending on the dimension $z$.
We also  use $\nabla_{x}$, $\nabla_{xx}$, $\di_x$ to denote the gradient, the Hessian, and the divergence, respectively, with respect to indicated variables. It should be clear from the context.

We will establish the LDP and describe explicitly
the rate function for the family $\{(X^\eps,\mu^\eps)\}_{\eps>0}$ in
$\C(\R_+,\R^d)\times\C_{\uparrow}(\R_+,\M(\R^l))$. The LDP and the rate function of $\{X^\eps\}_{\eps>0}$
are obtained directly by standard projections in the large deviations theory.
To proceed, we recall briefly the basic definitions
%and preliminaries in LDP theory.
of the LDP.
For further references, see
%The full construction can be found in
\cite{DS89,DZ98,LP92}.

\begin{deff}{\rm
We said	the family of $\{\PP^\eps\}_{\eps>0}$ in some metric space $\mathbb S$ enjoys the LDP with a
rate function $\I$ if the following conditions are satisfied:
	%\begin{itemize}
	1) $\I:\mathbb S\to[0,\infty]$ is inf-compact, that is, the
	level sets $\{z\in\mathbb S: \I(z)\leq L\}$ are compact in $\mathbb S$ for any $L>0$;
	and 2) for any open subset $G$ of $\mathbb S$,
	$$
	\liminf_{\eps\to 0}\eps\log\PP^\eps(G)\geq -\I(G):= -\inf_{z\in G}\I(f);
	$$
 and 3) for any closed subset $F$ of $\mathbb S$,
	$$
	\limsup_{\eps\to 0}\eps\log\PP^\eps(F)\leq -\I(F):= -\inf_{z\in F}\I(f).
	$$
	%\end{itemize}
	We say that a family of random elements of $\mathbb S$ obeys the
	LDP if the family of their laws obeys the LDP.
}\end{deff}
%\subsection{LDP: Fast Diffusion}\label{sec:dif}
Our main effort in this section is to consider  system \eqref{eq:F-setup} and to establish LDP
%for the averaging principle
for the family of the processes $\{(X^\eps,\mu^\eps)\}_{\eps>0}$ with $\mu^\eps$ being the empirical process associated with $Y^\eps$ as in \eqref{eq-mu}, where $(X^\eps,Y^\eps)$ is a solution of the second-order differential equation with random environment given in \eqref{eq:F-setup}.
Such a solution is defined as follows.

One can rewrite \eqref{eq:F-setup} as
\begin{equation}\label{eq-xpy}
\begin{cases}
\dot X^\eps_t=p^\eps_t, \quad X^\eps_0=x_0^\eps\in\R^d,\\
\eps^2\dot{p}^\eps_t=F^\eps_t(X^\eps_t,Y^\eps_t)-\lambda^\eps_t(X^\eps_t,Y^\eps_t)p^\eps_t,\quad p^\eps_0=x_1^\eps\in \R^d,\\
\dot Y^\eps_t=\dfrac 1\eps  b^\eps_t(X^\eps_t,Y^\eps_t)+\dfrac 1{\sqrt\eps}\sigma^\eps_t(X^\eps_t,Y^\eps_t)\dot W_t,\quad Y^\eps_0=y_0^\eps\in\R^l.
\end{cases}
\end{equation}
Recall that for each $\eps>0$,
the coefficients
$F^\eps_t(x,y):\R_+\times\R^d\times\R^l\to\R^d$,
$\lambda^\eps_t(x,y):\R_+\times\R^d\times\R^l\to\R$, $b^\eps_t(x,y):\R_+\times\R^d\times\R^l\to\R^l$, $\sigma^\eps_t(x,y):\R_+\times\R^d\times\R^l\to\R^{l\times m}$ are functions of $(t,x,y)$; $x_0^\eps,x_1^\eps\in\R^d, y_0^\eps\in\R^l$ are initial values that can be random.
Throughout the paper, we
assume that these functions are measurable and locally
bounded in $(t,x,y)$ such that the system of equations \eqref{eq-xpy} admits a weak
solution $(X^\eps,p^\eps,Y^\eps)$ with trajectories in $\C(\R_+,\R^d \times \R^d\times\R^l)$ for every initial condition
$(x_0^\eps,x_1^\eps,y_0^\eps)$;
and then the system of equations \eqref{eq:F-setup} admits a weak solution $(X^\eps,Y^\eps)$ with trajectories in $\C(\R_+,\R^d\times\R^l)$. More precisely, we assume that there exist a complete probability space
$(\Omega^\eps,\F^\eps,\PP^\eps)$ with filtration $\F^\eps = (\F^\eps_t,t\in\R_+)$, a Brownian motion
%process
$(W_t,t\in\R_+)$ with respect to $\F^\eps$, processes $X^\eps = (X^\eps_t,t \in\R_+)$, $p^\eps= (p^\eps_t,t\in\R_+)$, and $Y^\eps=(Y^\eps_t, t\in\R_+)$ that are $\F^\eps$-adapted and have continuous trajectories
%and
satisfying
the following
equations
% holds
\begin{equation*}
\begin{cases}
\displaystyle X^\eps_t=x^\eps_0+\int_0^tp^\eps_sds,\\[1ex]
 %\quad X^\eps_0=x_0^\eps\in\R^d,\\
\displaystyle p^\eps_t=x^\eps_1+\frac 1{\eps^2}\int_0^t\Big(F^\eps_s(X^\eps_s,Y^\eps_s)-\lambda^\eps_s(X^\eps_s,Y^\eps_s)p^\eps_s\Big)ds,\\[1ex]
%\quad p^\eps_0=x_1^\eps\in \R^d,\\
\displaystyle Y^\eps_t=y^\eps_0+\dfrac 1\eps  \int_0^tb^\eps_s(X^\eps_s,Y^\eps_s)ds+\dfrac 1{\sqrt\eps}\int_0^t\sigma^\eps_s(X^\eps_s,Y^\eps_s)d W_s,
%\quad Y^\eps_0=y_0^\eps\in\R^l.
\end{cases}
\end{equation*}
for all $t\in\R_+$, $\PP^\eps$-a.s.
 It is noted that it may not guarantee the uniqueness of the solution. [To ensure the uniqueness, one
may need to require further that the coefficients are Lipschitz continuous, which we do not assume here.]
%\para{Assumptions.}
 Next, we need some
 %following
 conditions, which are mild and natural, to establish the LDP for the family of coupled processes $\{X^\eps,\mu^\eps\}_{\eps>0}$.

\begin{asp}\label{asp-2}{\rm
	Assume that for all $L>0$ and $t>0$,
	\begin{equation}\label{eq-cond-F}
	%\begin{aligned}
\limsup_{\eps\to0}\sup_{s\in[0,t]}\sup_{y\in\R^l}\sup_{x\in\R^d:|x|\leq L}\big[|F^\eps_s(x,y)|
+|\lambda^\eps_s(x,y)|+|b^\eps_s(x,y)|+\|\Sigma^\eps_s(x,y)\|\big]<\infty,
%	\end{aligned}
	\end{equation}
	where $\Sigma^\eps_t(x,y):=\sigma^\eps_t(x,y)[\sigma^\eps_t(x,y)]^\top$,
%	\begin{equation}\label{eq-cond-b}
%	\limsup_{\eps\to0}\sup_{s\in[0,t]}\sup_{y\in\R^l}\sup_{x\in\R^d:|x|\leq L}|b^\eps_s(x,y)|<\infty,
%	\end{equation}
%	\begin{equation}\label{eq-cond-Sigma}
%	\limsup_{\eps\to0}\sup_{s\in[0,t]}\sup_{y\in\R^l}\sup_{x\in\R^d,|x|\leq L}\|\Sigma^\eps_s(x,y)\|<\infty,
%	\end{equation}
	\begin{equation}\label{eq-cond-F-growth}
	\limsup_{\eps\to0} \sup_{s\in[0,t]}\sup_{y\in\R^l}\sup_{x\in\R^d}\frac{x^\top F^\eps_s(x,y)}{(1+|x|^2)\lambda^\eps_s(x,y)}<\infty,
	\end{equation}
\begin{equation}\label{eq-cond-sta-1}
\lim_{M\to\infty}\limsup_{\eps\to0}\sup_{s\in[0,t]}\sup_{y\in\R^l, |y|\geq M}\sup_{x\in\R^n, |x|\leq L} \frac{[b^\eps_s(x,y)]^\top y}{|y|}<0,
\end{equation}
\begin{equation}\label{eq-cond-lambda1} \barray
\ad
\limsup_{\eps\to0}\sup_{s\in[0,\infty),y\in\R^l,x\in\R^d}
%\sup_{y\in\R^l}\sup_{x\in\R^d}
\big[|\nabla_s\lambda^\eps_s(x,y)|
+|\nabla_x\lambda^\eps_s(x,y)|\\
\aad \ +|\nabla_y\lambda^\eps_s(x,y)|
+\|\nabla_{yy}\lambda^\eps_s(x,y)\|\big]<\infty,\earray
\end{equation}
and
\begin{equation}\label{eq-cond-lambda2}
\liminf_{\eps\to0}\inf_{s\in[0,\infty], y\in\R^l, x\in\R^d}
%\inf_{y\in\R^l}\inf_{x\in\R^d}
\lambda^\eps_s(x,y)>\kappa_0>0.
\end{equation}
}\end{asp}

\begin{rem}{\rm
	The condition \eqref{eq-cond-F}\label{rem-11}
	%, \eqref{eq-cond-b}, \eqref{eq-cond-Sigma}
	is $($locally in $(t,x)$ and globally in $y)$ boundedness conditions of $F^\eps,b^\eps$ and $\Sigma^\eps$.
Note that
% condition
\eqref{eq-cond-F-growth} is a growth-rate condition, which is milder than linear growth %rate condition
of $\frac{F^\eps_s(x,y)}{\lambda^\eps_s(x,y)}$, e.g., $\frac{F^\eps_t(x,y)}{\lambda^\eps_s(x,y)}=\frac 1x$ satisfies this condition but is not linear growth.
	Moreover,
%that also
 it does not implies any growth-rate condition for $F^\eps_s(x,y)$.
	The condition \eqref{eq-cond-sta-1} is a stability condition, which in fact is needed for the ergodicity of the fast process.
		It is noted that we do not require
any Lipschitz continuity and growth-rate conditions for these coefficients.
	Lower boundedness and regularity conditions \eqref{eq-cond-lambda1} and \eqref{eq-cond-lambda2}
	of $\lambda^\eps_t(x,y)$ are natural and often used in
 the literature of mathematical physics; see, e.g., \cite{Cheng,Freidlin}.
}\end{rem}

	Assume that there are ``limit" measurable functions $F_t(x,y)$, $\lambda_t(x,y)$, $b_t(x,y)$,  and $\sigma_t(x,y)$ of
	the families of functions
	$F_t^\eps(x,y)$,
%[??$\sigma_t(x,y)$ of $F_t^\eps(x,y)$??]
$\lambda_t^\eps(x,y)$, $b_t^\eps(x,y)$, $\sigma_t^\eps(x,y)$ as $\eps\to 0$, respectively, in the sense that for all $t>0$ and $L>0$,
	\begin{equation}\label{eq-asp-conv}
	%\begin{aligned}
\barray
\ad	\lim_{\eps\to0} \sup_{s\in[0,t]}\sup_{y\in\R^l, |y|\leq L}\sup_{x\in\R^d, |x|\leq L}\Big[|F^\eps_s(x,y)%&
-F_s(x,y)|+|\lambda^\eps_s(x,y)-\lambda_s(x,y)|\\
\aad	\quad
%&
+|b^\eps_s(x,y)-b_s(x,y)|+\|\sigma^\eps_s(x,y)-\sigma_s(x,y)\|\Big]=0.
%	\end{aligned}
\earray
	\end{equation}
	
\begin{asp}\label{asp-3}{\rm
	Assume that the ``limit" function $b_t(x, y)$ is
	Lipschitz continuous in $y$ locally uniformly in $(t,x)$; 
	$b_t(x, y)$ and
	$\Sigma_t(x,y):=\sigma_t(x,y)[\sigma_t(x,y)]^\top$ are continuous in $x$ locally uniformly in $t$ and uniformly in $y$; $\Sigma_t(x,y)$ is
	of class $\C^1$ in $y$, with the first partial derivatives being bounded and Lipschitz continuous in $y$ locally uniformly in $(t, x)$, and $\text{\rm div}_y \Sigma_t(x, y)$ is continuous in $(x, y)$.
	The matrix $\Sigma_t(x, y)$ is positive definite uniformly in $y$ and
	locally uniformly in $(t, x)$. In addition, the ``limit" function $F_t(x,y)$ is locally Lipschitz continuous in $x$ locally uniformly in $t$ and uniformly in $y$.
	The conditions \eqref{eq-cond-lambda1} and \eqref{eq-cond-lambda2} hold for $\lambda_t$.
Moreover, for all
	$t>0$,
	\begin{equation}\label{eq-cond-sta-2}
	\lim_{|y|\to\infty}\sup_{s\in[0,t]}\sup_{x\in\R^d} \frac{[b_s(x,y)]^\top y}{|y|^2}<0.
	\end{equation}
}\end{asp}

\para{Rate function.}
Denote by $\G$ the collection of $(\varphi,\mu)$ such that the function $\varphi = (\varphi_t,t \in\R_+)\in\C(\R_+,\R^d)$ is absolutely continuous (with respect to the Lebesgue measure on $\R_+$) and
the function $\mu = (\mu_t,t \in\R_+)\in\C_{\uparrow}(\R_+,\M(\R^l))$, when considered as a measure on $\R_+\times \R^l$, is absolutely continuous (with respect to Lebesgue measure on
$\R_+\times \R^l$), i.e., $\mu(ds, dy) = m_s(y) dy ds$, and for almost all $s$, $m_s(y)$ (as a function of $y$) belongs to $\mathcal P(\R^l)$.

For $(\varphi,\mu)\in \G$, $\mu(ds,dy)=m_s(y)dyds$, define
\bea \ad
%\begin{aligned}
\I_1(\varphi,\mu)=\int_0^\infty\bigg[ \sup_{\beta\in\R^d}\beta^\top\bigg(\dot\varphi_s -\int_{\R^l}\frac{F_s(\varphi_s,y)}{\lambda_s(\varphi_s,y)}m_s(y)dy\bigg)\\
\aad\hspace{1.5cm}+\sup_{h\in \C_0^1(\R^l)}\int_{\R^l}\bigg([\nabla h(y)]^\top \Big(\frac12\di_y\big(\Sigma_s(\varphi_s,y)m_s(y)\big)-b_s(\varphi_s,y)m_s(y)\Big)\\
\aad\hspace{1.5cm}-\frac 12\|\nabla h(y)\|^2_{\Sigma_s(\varphi_s,y)}m_s(y)\bigg)dy\bigg]ds,
%\end{aligned}
\eea
and define $\I_1(\varphi,\mu)=\infty$ if $(\varphi,\mu)\notin\G$.

\begin{deff}{\rm
		The family of random variables with distributions $\{\PP^\eps\}_{\eps>0}$ is said to be exponentially tight in the space $\mathbb S$ if there exists an increasing sequence of compact sets $(K_L)_{L\geq 1}$ of $\mathbb S$ such that
$\lim_{L\to\infty}\limsup_{\eps\to0}\eps\log\PP^\eps(K_L)=-\infty.
		$}
\end{deff}

\begin{thm}\label{thm-main}
Assume that
Assumptions
{\rm\ref{asp-2}} and {\rm\ref{asp-3}} hold,
that
the family of initial values $\{x_0^\eps\}_{\eps>0}$ obeys the LDP in $\R^d$ with
a rate function $\I_0$, %and
that $$\limsup_{\eps\to0}\eps |x_1^\eps|<\infty \ \hbox{ a.s.,}$$ and that the family of initial values $\{y_0^\eps\}_{\eps>0}$ is exponentially tight in $\R^l$.
Then the family $\{(X^\eps,\mu^\eps)\}_{\eps>0}$ obtained from \eqref{eq:F-setup} obeys the LDP in $\C(\R_+,\R^d)\times \C_{\uparrow}(\R_+,\M(\R^l))$ with rate function $\I$ defined as
		$\I(\varphi,\mu) =
        \I_0(\varphi_0) + \I_1(\varphi,\mu),\text{ if }(\varphi, \mu) \in\G,$
$\I(\varphi,\mu) =	\infty,\text{ otherwise}.$
\end{thm}

\begin{cor}
	Under the hypotheses of Theorem \ref{thm-main}, the family $\{X^\eps\}_{\eps>0}$ satisfies the
	LDP in $\C(\R_+,\R^d)$ with the
rate function $\I_X$ defined
	by $$\I_X(\varphi)=\inf_{\mu\in\C_{\uparrow}(\R_+,\M(\R^l))}\I(\varphi,\mu).$$
	As an alternative representation,
	if function $\varphi= (\varphi_t,t\in\R_+) \in\C(\R_+,\R^d)$ is absolutely continuous
	with respect to Lebesgue measure on $\R_+$, then
\begin{equation}\label{eq-rate}
\begin{aligned}
\ad\! \I_X(\varphi)=\I_0(\varphi_0)+\int_0^\infty \sup_{\beta\in\R^d}\bigg[\beta^\top\dot\varphi_s -\sup_{m\in\mathcal P(\R^l)}\bigg(\beta^\top\int_{\R^l}\frac{F_s(\varphi_s,y)}{\lambda_s(\varphi_s,y)}m(y)dy\\
\aad \qquad\quad +\sup_{h\in \C_0^1(\R^l)}\int_{\R^l}\!\bigg([\nabla h(y)]^\top \Big(\frac12\di_y\big(\Sigma_s(\varphi_s,y)m(y)\big)-b_s(\varphi_s,y)m(y)\Big)\\
\aad \qquad\quad
-\frac 12\|\nabla h(y)\|^2_{\Sigma_s(\varphi_s,y)}m(y)\bigg)dy\bigg)\bigg]ds,
\end{aligned}
\end{equation}
	otherwise, $\I_X(\varphi) = \infty$.
\end{cor}

\subsubsection{Zero points of $\I(\varphi,\mu)$, averaging principle of
\eqref{eq:F-setup},
%fast-slow second-order systems with fast diffusion,
and its large deviations analysis}\label{sec:num}
We start with an intuitive
discussion on the behavior of
\eqref{eq:F-setup}
as $\eps\to0$.
 Intuitively, there are two phases as $\eps\to 0$.
%At the first phase,
 First, $\eps^2$ goes to 0.
 %(much faster) to $0$ first.
At this phase, $X^\eps_t$ is
close to the solution of the following associated first-order equation (or the over-damped equation in the language of statistical physics)
\begin{equation}\label{eq-11111}
\begin{cases}
0=F^\eps_t(\bar X^\eps_t,Y^{\eps,\bar X}_t)-\lambda^\eps_t(\bar X^\eps_t,Y^{\eps,\bar X}_t)\dot{\bar X}^\eps_t,\quad \bar X^\eps_0=x_0^\eps \in\R^d,\\[1ex]
\dot Y^{\eps,\bar X}_t=\dfrac 1\eps  b^\eps_t(\bar X^\eps_t,Y^{\eps,\bar X}_t)+\dfrac 1{\sqrt\eps}\sigma^\eps_t(\bar X^\eps_t,Y^{\eps,\bar X}_t)\dot W_t,\quad Y^{\eps,\bar X}_0=y_0^\eps\in\R^l.
\end{cases}
\end{equation}
Next, $Y^{\eps,\bar X}_t$  converges to its invariant distribution as $\eps\to0$.
More precisely, if we let $\wdt Y^{\bar X}_t:=Y^{\eps,\bar X}_{t\eps}$, then
$$
\begin{cases}
\dot{\bar X}^\eps_t=\dfrac{F^\eps_t(\bar X^\eps_t,\wdt Y^{\bar X}_{t/\eps})}{\lambda^\eps_t(\bar X^\eps_t,\wdt Y^{\bar X}_{t/\eps})},\quad\bar X^\eps_0=x^\eps_0\in\R^d,\\
\dot {\wdt Y}^{\bar X}_t=b^\eps_{t\eps}(\bar X^\eps_{t\eps},\wdt Y^{\bar X}_t)+ \sqrt\eps\sigma^\eps_{t\eps}(\bar X^\eps_{t\eps},\wdt Y^{\bar X}_t)\dot {\wdt W}_t,\quad \wdt Y^{\eps,\bar X}_0=y_0^\eps\in\R^l,
\end{cases}
$$
where ${\wdt W}_t$ is another standard Brownian motion.
As a consequence, because $\wdt Y^{\bar X}_{t/\eps}$ will come to and stay close to its invariant measure as $\eps\to 0$, $\bar X^\eps_t$ will tend to $\bar X_t$,
%where $\bar X_t$ is
the solution of the following differential equation
$$
\dot {\bar X}_t=\bar {F/\lambda}_t(\bar X_t),\quad  \bar X_0=\bar x_0,
$$
where $\bar x_0$
is the limit of
%=\lim_{\eps\to0}
$x_0^\eps$,
and
$$
\bar {F/\lambda}_{t_1}(x):=\int_{\R^l}\frac{F_{t_1}(x,y)}{\lambda_{t_1}(x,y)}\bar \nu^{t_1,x}(dy),
$$
and for each fixed $(t_1,x)$, $\bar\nu^{t_1,x}(dy)$ is the invariant measure of the following stochastic differential equation
$$
\dot{\wdt Y_t}=b_{t_1}(x,\wdt Y_t)+\sigma_{t_1}(x,\wdt Y_t)\dot{\wdt W}_t.
$$
The convergence of $X^\eps_t$ to $\bar X_t$ as $\eps\to 0$ forms %somewhat
an averaging principle of \eqref{eq:F-setup}.
However, not only are we interested in the convergence of $X^\eps$ to $\bar X_t$, but also the tail probability of this convergence, i.e., the rate of the convergence of the probability of the event $\{\|X^\eps-\bar X\|>\eta\}$ to $0$, for any $\eta>0$.
We show that the convergence is exponentially fast.
The answer to these questions can be obtained from the LDP for $\{X^\eps_t\}_{\eps>0}$ and
explicit representations of the rate function.

To proceed, we apply our results to make the above intuition rigorous.
It is shown later that
%(see, e.g., \cite[Lemma 6.7]{Puh16})
%in the main paper
that $\I_1(\varphi, \mu) = 0$ provided that a.e.
$$
\dot\varphi_s = \int_{\R^l} \frac{F_s(\varphi_s,y)}{\lambda_s(\varphi_s,y)}m_s(y) dy,
$$
and
$m_s(y)$ satisfies the following equation
\begin{equation}\label{eq-m}
\int_{\R^l}\Big(\frac12 \tr(\Sigma_s(\varphi_s,y)[\nabla_{yy}h(y)]) + [\nabla_yh(y)]^\top b_s(\varphi_s,y)\Big) m_s(y) dy = 0, \text{ for all }h\in\C^{\infty}_0(\R^l),
\end{equation}
and
$\I_0(\varphi_0) = 0$. Alternatively, $m_s(\cdot)$ is the invariant density of the diffusion process
with the
%infinitesimal
drift $b_s(\varphi_s,\cdot)$ and the diffusion matrix $\Sigma_s(\varphi_s,\cdot)$.
Therefore, as $\eps\to0$, the trajectories of $\{X^\eps\}_{\eps>0}$ %will
hover
%focus
around $\bar X$
%defined as below
with exponential tail probability, where $\bar X$ is defined as the solution of the following ODE
\begin{equation}\label{eq-barX}
\dot{\bar X}_t=\bar{F/\lambda}_t(\bar X_t),\quad\bar X_0=\bar x_0,
\end{equation}
with
$$
\bar{F/\lambda}_t(x):=\int_{\R^l}\frac{F_t(x,y)}{\lambda_t(x,y)}\bar m_t(y)dy,
$$
and $\bar m_t(\cdot)$ satisfies equation \eqref{eq-m} and $\bar x_0$
%is such that
satisfying
$\I_0(\bar x_0)=0$.
%More specifically,
%specially,
Let
$$B_\eta^c(\bar X):=\Big\{\varphi\in\C(\R_+,\R^d):\|\varphi_t-\bar X_t\|_{\C(\R_+,\R^d)}:=\sum_{n=1}^\infty
\frac 1{2^n} \big(1 \wedge \sup_{t\leq n}|\varphi_t-\bar X_t| \big)\geq \eta\Big\},$$
the LDP established in this paper
%in Section \ref{sec:re}
implies that
$$
\PP^\eps(X^\eps\in B_\eta^c(\bar X))\sim e^{-\frac 1{\eps}\I_X(B_\eta^c(\bar X))},
$$
where
$\I_X(B_\eta^c(\bar X))=\inf_{\varphi\in B_\eta^c(\bar X)}\I_X(\varphi)$.
If we assume that $\bar X$ is the unique solution of \eqref{eq-barX}, it is the unique solution of $\I_X(\varphi)=0$.
As a result, $\I_X(B_\eta^c(\bar X))>0$. Indeed, if $\I_X(B_\eta^c(\bar X))=0$, there exists $\{\varphi_k\}_{k=1}^\infty\subset B_\eta^c(\bar X)$ such that $\lim_{k\to\infty}\I_X(\varphi_k)=0$. Because of that $\I_X$ is a rate function, there exists a convergent subsequence
%subconsequence
(still denoted by $\varphi_k$) of $\{\varphi_k\}$ with limit denoted by $\bar\varphi\in B^c_\eta(\bar X)$. Since $\I_X$ is lower semi-continuous, $0\leq \I_X(\bar\varphi)=\I_X(\lim_{k\to\infty}\varphi_k)\leq \lim_{k\to\infty}\I_X(\varphi_k)=0$. It leads to $\I_X(\bar\varphi)=0$, which is a contradiction.
Because $\I_X(B_\eta^c(\bar X))>0$,
$\PP(\|X^\eps-\bar X\|>\eta)\to 0$
exponentially fast for any $\eta>0$.

	\begin{rem}
	{\rm
		In Section \ref{sec:num}, we illustrate that from our LDP result, we can establish the averaging principle of \eqref{eq:F-setup} with exponentially convergence rate in the sense that $X^\eps$ converges to $\bar X$ (of \eqref{eq-barX}) with exponential tail probability, i.e., for any $\eta>0$, $\PP(\|X^\eps-\bar X\|>\eta)\to 0$
		exponentially fast.
From a different angle, references \cite{Xu22,XW23}  treated convergence rate for averaging principles of different problems  using certain moments. In this process, just as treating $L_2$ or $L_p$  convergence rates in numerical approximation of stochastic differential equations, global Lipschitz conditions are needed.
For our second-order equations, the Lipschitz continuity need not be assumed. This is an advantage.
}
\end{rem}

\subsubsection{Alternative representations of
$\I(\varphi,\mu)$}
One can write the rate function $\I(\varphi,\mu)$ as
\bea
\I(\varphi,\mu)=\I_0(\varphi_0)+\int_0^\infty\Big[ \sup_{\beta\in\R^d}\beta^\top\Big(\dot\varphi_s-\int_{\R^l}\frac{F_s(\varphi_s,y)}{\lambda_s(\varphi_s,y)}\nu_s(dy)\Big)+\JJ_{s,\varphi_s}(\nu_s)\Big]ds,
\eea
where $\nu_s(dy)=m_s(y)dy$ and
\bea \ad
\JJ_{s,\varphi_s}(\nu_s):=\sup_{h\in \C_0^1(\R^l)}\int_{\R^l}\bigg([\nabla h(y)]^\top \Big(\frac12\di_x\big(\Sigma_s(\varphi_s,y)m_s(y)\big)-b_s(\varphi_s,y)m_s(y)\Big)\\
\aad \hspace{1.1in}-\frac 12\|\nabla h(y)\|^2_{\Sigma_s(\varphi_s,y)}m_s(y)\bigg)dy.
\eea
In fact, for each $(s,x)\in\R_+\times\R^d$,
$\JJ_{s,x}(\nu)$ is the large deviations rate function for the empirical measures
$$\nu^{s,x}_t(dy) = \frac 1t\int_0^t \1_{dx}(\wdt Y_r^{s,x})dr$$ for rate $\eps=1/t$ as $t\to\infty$ and
$$
\dot{\wdt Y}^{s,x}_t=b_s(x,\wdt Y^{s,x}_t)+\sigma_s(x,\wdt Y^{s,x}_t)d\wdt W_t;
$$
see \cite[Section 2, Corollary 2.2 and 2.3]{Puh16}.

Moreover, if $\I(\varphi,\mu)$ is finite,
it is necessary that
$$
\dot\varphi_s = \int_{\R^l} \frac{F_s(\varphi_s,y)m_s(y)}{\lambda_s(\varphi_s,y)}dy \ \hbox{ a.e.,}
$$
and in this case, we have
$
\I(\varphi,\mu)=\I_0(\varphi_0)+\int_0^\infty \JJ_{s,\varphi_s}(\nu_s)ds.
$
On the other hand, one can also write the rate function (see \cite[Section 2, Proposition 2.1]{Puh16})
\begin{equation}\label{rate-alternative}
\I(\varphi,\mu) = \I_0(\varphi_0)+
\frac12\int_0^\infty\int_{\R^l}\Big\|\frac{\nabla_y m_s(y)}{2m_s(y)}-\J_{s,m_s(\cdot),\varphi_s}(y)\Big\|_{\Sigma_s(\varphi_s,y)}m_s(y) dy ds,
\end{equation}
with $\mu(dy,ds)=m_s(y)dyds$, where for each $s\in\R_+$, each function $m_s(\cdot)$ belongs to $\mathcal P(\R^l)$, and
 %In the above,
 $\J_{t,m(\cdot),u}$ is a function defined as follows.
%With
Denote $\mathbb L^2(\R^l,\R^l, \Sigma_s(\varphi_s,y),m_s(y)$ $dy)$
%denoting
the Hilbert space of all $\R^d$-valued functions (of $y$) in $\R^l$
%endowed
with
%the
 norm $$\|f\|_{\Sigma,m}^2=\int_{\R^l}\|f(y)\|^2_{\Sigma_s(\varphi_s,y)}m_s(y)dy$$ and
$\mathbb L^2_{\rm loc}(\R^l,\R^l, \Sigma_t(x,y),m(y) dy)$  the space
%consists
consisting
of functions whose products with arbitrary $\C_0^\infty$-functions belong to $\mathbb L^2(\R^l,\R^l, \Sigma_t(x,y),m(y) dy)$,
then $\J_{t,m(\cdot),u}$ is defined as a function
of $y$ by
$$
\J_{t,m(\cdot),u}(y)=\Pi_{ \Sigma_t(x,\cdot),m(\cdot)}(\Sigma_t(x,y)^{-1}(b_t(x,y) - \di_x \Sigma_t(x,y)/2)),
$$
where $\Pi_{ \Sigma_t(x,\cdot),m(\cdot)}$ maps $\phi(y)\in \mathbb L^2_{\rm loc}(\R^l,\R^l, \Sigma_t(x,y),m(y) dy)$
to
$\Pi_{ \Sigma_t(x,\cdot),m(\cdot)}\phi(y)$, which belongs to $\mathbb L_0^{1,2}(\R^l,\R^l, \Sigma_t(x,y),m(y) dy)$ and satisfies that for all
$h\in\C^\infty_0 (\R^l)$,
$$
\int_{\R^l} [\nabla h(y)]^\top \Sigma_t(x,y)\Pi_{ \Sigma_t(x,\cdot),m(\cdot)}\phi(y)m(y)dy = \int_{\R^l} [\nabla h(y)]^\top \Sigma_t(x,y)\phi(y)m(y) dy.
$$
If $\phi(y)\in \mathbb L^2(\R^l,\R^l,\Sigma_t(x,y),m(y) dy)$, then $\Pi_{ \Sigma_t(x,\cdot),m(\cdot)}\phi(y)$ is nothing than the orthogonal projection of $\phi$ onto $\mathbb L_0^{1,2}(\R^l,\R^l, \Sigma_t(x,y),m(y) dy)$.

%\begin{rem}
	In fact, $\I(\varphi,\mu)$ is defined similarly to the rate function of the family of processes $\{(\bar X^\eps,\mu^{\eps,\bar X})\}_{\eps>0}$,
	%(as in \cite{Puh16}),
	 where
	$
	\mu^{\eps,\bar X}_t(\A)=\int_0^t\1_{\A}(Y^{\eps,\bar X}_s)ds,\quad \A\in\mathcal B(\R^l),
	$
	and $(\bar X^\eps_t,Y^{\eps,\bar X}_t)$ is the solution of
	the following equation
	$$
	\begin{cases}
	\dot{\bar X}^\eps_t=\dfrac{F^\eps_t(\bar X^\eps_t,Y^{\eps,\bar X}_t)}{\lambda^\eps_t(\bar X^\eps_t,Y^{\eps,\bar X}_t)},\quad \bar X^\eps_0=x_0^\eps \in\R^d,\\[3ex]
	\dot Y^{\eps,\bar X}_t=\dfrac 1\eps  b^\eps_t(\bar X^\eps_t,Y^{\eps,\bar X}_t)+\dfrac 1{\sqrt\eps}\sigma^\eps_t(\bar X^\eps_t,Y^{\eps,\bar X}_t)\dot W_t,\quad Y^\eps_0=y_0^\eps\in\R^l.
	\end{cases}
	$$
%\end{rem}

\subsection{Proof of Theorem \ref{thm-main}}\label{sec:prof}
This section is devoted to proving Theorem \ref{thm-main}.
 In the proof, we use
  $C$ to represent a
 %(universal)
 generic positive
 %generative
 constant that is
 %are
 independent of $\eps$. The value $C$ may change at different appearances;
  we will specify which parameters it depends on if it is necessary.

\subsubsection{A Road Map For the Development of Our Analysis and Proof}
To make the proof be more readable, we first provide a road map and then the details will be illustrated in following sections.
The proof of the LDP of $\{(X^\eps,\mu^\eps)\}_{\eps>0}$ is based on the approach of \cite{Puh16}, which relies on the properties that if a family of random elements is exponentially tight then it is sequentially large deviation (LD)
relatively compact,
i.e., any subsequence contains a further subsequence
enjoying the LDP with some rate function.
The remaining work is done by carefully identifying the rate functions.
Specifically, the details are as follows.

{\bf Step 1:} The exponential tightness of $\{(X^\eps,\mu^\eps)\}_{\eps>0}$ is proved in Section \ref{sec:exp} by applying the (extended) Puhalskii's criteria.
Particularly, dealing with $X^\eps$, we  prove \eqref{p-cre-11}, which shows that $\{X^\eps\}$ cannot be large with exponentially small probability and \eqref{p-cre-12},
%which verifies
leading to
needed continuity properties.
To prove these, a first step is to use Lemma \ref{lm-intbypart} to
%handle
deal
with the large factor $\frac 1{\eps^2}$.
Then, taking advantages of the martingale property of stochastic integrals
%will allow
enables
us to establish desired estimates.
It is similar for $\mu^\eps$.

{\bf Step 2:} After proving exponentially tightness of $\{(X^\eps,\mu^\eps)\}$, thanks to Proposition \ref{prop21},
$\{(X^\eps,\mu^\eps)\}$ is sequentially LP relatively compact (Definition \ref{def:LD}).
Therefore, the second step is devoted to
identifying the large deviations (LD) limit points.
More precisely, let $\hat\I$ be a large deviations limit rate functions or
LD limit points of
$\{(X^\eps,\mu^\eps)\}_{\eps>0}$ (i.e., a rate function of some subsequence of $\{(X^\eps,\mu^\eps)\}_{\eps>0}$ that obeys the LDP) and we
prove that $\hat \I=\I$, ($\I$ is the rate function defined in Section \ref{sec:re}).
Details for this step is as follows.

{\bf Step 2a:} We introduce another characterization of the rate function in Section \ref{sec:iden}. Precisely,
for each step function $\beta(s)$, each $f(t,x,y)$
real-valued $\C^{1,2,2}(\R_+\times \R^d \times \R^l)$-function with compact support in $y$ locally uniformly in $(t, x)$,
define $\Phi_t^{\beta,f}$ as in \eqref{eq-phi-1}.
Then, $\I^*$ is defined as the supremum of $\Phi_{t\wedge\tau}^{\beta,f}$ over $\beta$, $f$, and stopping times $\tau$ (see \eqref{I*}).
Later, as a byproduct of the study of the regularity of $\I^*$ (in Section \ref{sec:iden-com}), it is shown that $\I^*=\I$, thanks to their
alternative representations \eqref{rate-alternative} and \eqref{I*-alternative}.

{\bf Step 2b:} In Section \ref{sec:IsIhat}, we
prove the lower bound of LD limit, i.e., $\I^*\leq \hat \I$ for any $(\varphi,\mu)$ or
$\sup_{(\varphi,\mu)\in\C(\R_+,\R^d) \times \C_{\uparrow}(\R_+,
	\M(\R^l))}\Big(\Phi^{\beta,f}_{t\wedge\tau(\varphi,\mu)}(\varphi,\mu)-\hat \I(\varphi,\mu)\Big)=0.
$
To prove this claim, using Lemma \ref{lm-vah} (or \cite[Theorem 2.1.10]{DS89}), it suffices to show that
$
\E^\eps\exp\Big\{\frac 1\eps\Phi^{\eps,\beta,f}_{t\wedge \tau(X^\eps,\mu^\eps)}(X^\eps,\mu^\eps)\Big\}=1,$
and
$
\Phi^{\eps,\beta,f}_{t\wedge \tau(\varphi,\mu)}(\varphi,\mu)\to\Phi^{\beta,f}_{t\wedge \tau(\varphi,\mu)}(\varphi,\mu)
$
as $\eps\to0$ uniformly in compact sets; see Theorem \ref{thm-hatI}.

{\bf Step 2c:} Section \ref{sec:iden-com} is devoted to the proof of upper bound of LD limits, i.e., $\hat \I\leq \I^*$.
We first show this claim at sufficiently regular (dense) points; see Theorem \ref{thm-hat-m}. Then, it is shown for arbitrary points by an approximation
%via
using regular points and applying some continuity properties of rate functions.

\begin{rem}
	{\rm Although our proof is inspired by
the approach of \cite{Puh16},  let us  highlight briefly some differences in our works as follows.
	(i) System \eqref{eq:F-setup} is second-order while \cite{Puh16} studied first-order equation. We do not assume any Lipschitz continuities for coefficients. Therefore, it is not possible to get a ``good estimate" (e.g., exponentially close) between \eqref{eq:F-setup} and its corresponding overdampped first-order equation.
	(ii) To identify the rate function in step 2,
	when defining $\I^*$, we need to narrow the space taking the supremum to control terms containing the derivative $p^\eps$ of $X^\eps$ (which is
%appears
due to considering second-order system), specially the integral involving $p^\eps$ and the diffusion part of the fast process (see \eqref{eq-betadx-1}); see Remark \ref{rem-25}.}
\end{rem}

\subsubsection{Exponential Tightness of $\{(X^\eps,\mu^\eps)\}_{\eps>0}$}\label{sec:exp}
In this section, we establish
the exponential tightness of $\{(X^\eps,\mu^\eps)\}_{\eps>0}$ in $\C(\R_+,\R^l)\times \C_{\uparrow}(\R_+,\M(\R^l))$.
To be self-contained, we recall these preliminaries below;  see
\cite{DS89,DZ98,LP92} for more detail.

\begin{deff}\label{def:LD}{\rm
		The family $\{\PP^\eps\}_{\eps>0}$ is said to be sequentially LD %large deviations (LD)
		relatively
		compact if any subsequence $\{\PP^{\eps_k}\}_{k\geq 1}$ of $\{\PP^\eps\}_{\eps>0}$ contains a further
		subsequence $\{\PP^{\eps_{k_j}}\}_{j\geq 1}$ which satisfies the LDP with some
		large deviations
		rate
		function as $j\to\infty$.
		We say that a family of random elements of $\mathbb S$ is sequentially LD relatively compact (resp. exponentially tight)  if
		%the family of
		their laws have the indicated property.
}\end{deff}

\begin{prop}\label{prop21}
	{\rm(\cite[Theorem 4.1]{Puh16})}
	If a family $\{\PP^\eps\}_{\eps>0}$ is exponentially tight then it is sequentially LD relatively compact.
\end{prop}

To start, we introduce the following
technical lemma, which will be used often in some calculations
in
this section.

\begin{lm}\label{lm-intbypart}
	For real-valued continuous function $g(s)$, and real-valued continuously differentiable function $u(s)$, and real-valued It\^o process $w(s)$, $w(s)>0\; \forall s$ with the quadratic variation denoted by $\langle dw,dw\rangle_s$, we have the following identity
	\begin{equation}\label{eq-intbypart}
		%\begin{aligned}
		\barray \ad
		\frac 1{\eps^2}\int_0^t u(s)\int_0^s e^{-\frac 1{\eps^2}\int_r^sw(r')dr'}g(r)drds
		\\
		
		\ad \ =\int_0^t\frac{u(s)g(s)}{w(s)}ds+\int_0^t\frac{\nabla_su(s)}{w(s)}\left(\int_0^s e^{-\frac 1{\eps^2}\int_r^sw(r')dr'}g(r)dr\right)ds\\
		&-\frac{u(t)}{w(t)}\int_0^t e^{-\frac 1{\eps^2}\int_s^tw(r)dr}g(s)ds-\int_0^t \frac{u(s)}{w(s)^2}\left(\int_0^se^{-\frac 1{\eps^2}\int_r^sw(r')dr'}g(r)dr\right)dw(s)\\
		%&
		\aad \ +\int_0^t \frac{u(s)}{w(s)^3}\left(\int_0^se^{-\frac 1{\eps^2}\int_r^sw(r')dr'}g(r)dr\right)\langle dw,dw\rangle_s.
		%	\end{aligned}
	\earray	\end{equation}
Moreover, the identity \eqref{eq-intbypart} still holds if $g$ is $\R^d$-valued function and $u$ can be either $\R$-valued or $\R^d$-valued $($with the operations corresponding to $u$ and $g$ being understood as the inner product in $\R^d)$.
%		Moreover, \eqref{eq-intbypart} also holds if $v(s)$ is replaced by a It\^o process.
\end{lm}

\begin{proof}
Using integration by parts for
$$u(s)\int_0^se^{\frac 1{\eps^2}\int_0^rw(r')dr'}g(r)dr\text{ and } \frac{e^{-\frac1{\eps^2}\int_0^sw(r)dr}}{w(s)},$$
\eqref{eq-intbypart} follows from standard calculations.
\end{proof}

\begin{thm}
	Suppose that Assumption {\rm\ref{asp-2}} holds,
	that the family $\{(x_0^\eps,y_0^\eps)\}_{\eps>0}$ is exponentially tight, and that $\limsup_{\eps\to0}\eps x_1^\eps<\infty$ \a.s Then %the family
$\{(X^\eps,\mu^\eps)\}_{\eps>0}$ obtained from \eqref{eq:F-setup} is exponentially tight and
%then,
sequentially LD relatively compact in
	$\C(\R_+,\R^l)\times \C_{\uparrow}(\R_+,\M(\R^l))$.
\end{thm}

Since $\C(\R_+,\R^l)\times \C_{\uparrow}(\R_+,\M(\R^l))$ is a closed
subset of $\C(\R_+,\R^d)\times \C(\R_+,\M(\R^l))$ and $\PP^\eps\big((X^\eps,\mu^\eps)$ $\in\C(\R_+,\R^d) \times \C_{\uparrow}(\R_+,\M(\R^l))\big) = 1$, it is sufficient to prove that the family $\{(X^\eps,\mu^\eps)\}_{\eps> 0}$ is exponentially tight in $\C(\R_+,\R^d)\times\C(\R_+,\M(\R^l))$.
To prove that, it suffices to verify $\{(X^\eps,\mu^\eps)\}_{\eps>0}$ satisfying the (extended) Puhalskii's criteria (see
%Liptser and Puhalskii
\cite[Theorem 3.1]{LP92} and
\cite[Remark 4.2]{FeKu},
%or see also \cite[Lemma 4.1]{Puh16}),
namely, $\forall \ell,t>0$,
%, see \eqref{p-cre-1} and \eqref{p-cre-2}.
\begin{equation}\label{p-cre-11}
\lim_{L\to\infty}\limsup_{\eps\to 0}\eps\log \PP^\eps\Big(\sup_{s\in[0,t]}|X^\eps_s|>L\Big)=-\infty,
\end{equation}
\vspace{-10pt}
\begin{equation}\label{p-cre-12}
\lim_{\delta\to 0}\limsup_{\eps\to0}\sup_{s\in [0,t]}\eps\log\PP^\eps\Big(\sup_{s\leq s_1\leq s+\delta}\left|X^\eps_{s_1}-X^\eps_s\right|>\ell\Big)=-\infty,
\end{equation}
\vspace{-10pt}
\begin{equation}\label{p-cre-21}
\lim_{L\to\infty}
\limsup_{\eps\to0}
\eps\log\PP^\eps \Big(\mu^\eps\big([0,t],\{y\in \R^l : |y| > L\}\big)> \ell\Big) =-\infty,%\quad \text{and}
\end{equation}
\vspace{-10pt}
\begin{equation}\label{p-cre-22}
\lim_{\delta\to0}
\limsup_{\eps\to0}
\sup_{s\in[0,t]}
\eps\log\PP^\eps\Big( \sup_{s_1\in[s,s+\delta]}d(\mu^\eps_{s_1},\mu^\eps_s)> \ell \Big)= -\infty.
\end{equation}

\begin{rem}{\rm
	In general, \eqref{p-cre-11}-\eqref{p-cre-22} only imply the sequentially exponential
	tightness (i.e., any subsequence is exponentially tight). However, because
%due to
$(X^\eps,\mu^\eps)$ is continuous in $\eps$ in distribution, it is true
	that $(X^\eps,\mu^\eps)$ is exponentially tight although in our proof, only the sequentially exponential
	tightness is needed.
}\end{rem}

From equation \eqref{eq-xpy}, by the variation of parameter formula, we  obtain
\begin{equation}\label{formulap}
%\begin{aligned}
p^\eps_t=x_1^\eps e^{-A_\eps(t)}+\dfrac 1{\eps^2}\int_0^te^{-A_\eps(t,s)}F^\eps_s(X^\eps_s,Y^\eps_s)ds,
%\end{aligned}
\end{equation}
where for any $0\leq s\leq t, \eps>0$,
$A_\eps(t,s):=\dfrac 1{\eps^2}\int_s^t\lambda^\eps_r(X_r^\eps,Y^\eps_r)dr,\quad A_\eps(t)=A_\eps(t,0).$

\begin{proof}[Proof of \eqref{p-cre-11}]
	It is readily seen that
	\begin{equation}\label{eq-lnX-1}\barray
%	\begin{aligned}
	\ad\ln (1+|X^\eps_t|^2)-\ln (1+|x_0^\eps|^2)=\int_0^t \frac{2(X^\eps_s)^\top p^\eps_s}{1+|X^\eps_s|^2}ds\\
	=\ad\int_0^t \frac{2(X^\eps_s)^\top x_1^\eps}{1+|X^\eps_s|^2}e^{-A_\eps(s)}ds+\frac 1{\eps^2}\int_0^t\frac{2}{1+|X^\eps_s|^2}(X^\eps_s)^\top\Big(\int_0^s e^{-A_\eps(s,r)}F^\eps_r(X^\eps_r,Y^\eps_r)dr\Big).
\earray
%	\end{aligned}
	\end{equation}
Denote
$
v^\eps_t=\frac{2}{1+|X^\eps_t|^2}(X^\eps_t)^\top\Big(\int_0^t e^{-A_\eps(t,r)}F^\eps_r(X^\eps_r,Y^\eps_r)dr\Big).
$	
	We have from \eqref{eq-lnX-1} and Lemma \ref{lm-intbypart} that
\begin{equation}\label{eq-lnX-11}
\begin{aligned}
%\barray
\ad \ln (1+|X^\eps_t|^2)-\ln (1+|x_0^\eps|^2)\\
\ad =\int_0^t \frac{2(X^\eps_s)^\top x_1^\eps}{1+|X^\eps_s|^2}e^{-A_\eps(s)}ds+\int_0^t\frac{2(X^\eps_s)^\top F^\eps_s(X^\eps_s,Y^\eps_s)}{(1+|X^\eps_s|^2)\lambda^\eps_s(X^\eps_s,Y^\eps_s)}ds\\
\ad \ +\int_0^t\frac{2}{\lambda^\eps_s(X^\eps_s,Y^\eps_s)}\left(\frac{p^\eps_s}{1+|X^\eps_s|^2}-\frac{2(X^\eps_s[X^\eps_s]^\top)p^\eps_s}{(1+|X^\eps_s|^2)^2}\right)^\top\left(\int_0^s e^{-A_\eps(s,r)}F^\eps_r(X^\eps_r,Y^\eps_r)dr\right)ds\\
\ad \ -\frac{v^\eps_t}{\lambda^\eps_t(X^\eps_t,Y^\eps_t)}-\int_0^t\frac{v^\eps_s}{[\lambda^\eps_s(X^\eps_s,Y^\eps_s)]}d\lambda^\eps(X^\eps_s,Y^\eps_s)\\
\ad\
+ \int_0^t \frac{v^\eps_s}{[\lambda^\eps_s(X^\eps_s,Y^\eps_s)]^3}\langle d\lambda^\eps_s(X^\eps_s,Y^\eps_s),d\lambda^\eps_s(X^\eps_s,Y^\eps_s)\rangle_s.
%\earray
\end{aligned}
\end{equation}	
	Combining \eqref{eq-lnX-11} and the It\^o Lemma, one has
\bea
	\begin{aligned}
	\ad	
\ln (1+|X^\eps_t|^2)-\ln (1+|x_0^\eps|^2)\\
		\aad =\int_0^t \frac{2(X^\eps_s)^\top x_1^\eps}{1+|X^\eps_s|^2}e^{-A_\eps(s)}ds+\int_0^t\frac{2(X^\eps_s)^\top F^\eps_s(X^\eps_s,Y^\eps_s)}{(1+|X^\eps_s|^2)\lambda^\eps_s(X^\eps_s,Y^\eps_s)}ds\\ \aad +\int_0^t\frac{2}{\lambda^\eps_s(X^\eps_s,Y^\eps_s)}\left(\frac{p^\eps_s}{1+|X^\eps_s|^2}-\frac{2(X^\eps_s[X^\eps_s]^\top)p^\eps_s}{(1+|X^\eps_s|^2)^2}\right)^\top\Big(\int_0^s e^{-A_\eps(s,r)}F^\eps_r(X^\eps_r,Y^\eps_r)dr\Big)ds\\
%		&-\frac{2(X^\eps_t)^\top}{(1+|X^\eps_t|^2)\lambda^\eps_t(X^\eps_t,Y^\eps_t)}\int_0^t e^{-A_\eps(t,s)}F^\eps_s(X^\eps_s,Y^\eps_s)ds\\
		\aad -\frac{v^\eps_t}{\lambda^\eps_t(X^\eps_t,Y^\eps_t)}
%		&-\int_0^t \frac{2 }{[\lambda^\eps_s(X^\eps_s,Y^\eps_s)]^2(1+|X^\eps_s|^2)}(X^\eps_s)^\top\left(\int_0^s e^{-A_\eps(s,r)}F^\eps_r(X^\eps_r,Y^\eps_r)dr\right)[\nabla_X\lambda(X^\eps_s,Y^\eps_s)]^\top p^\eps_sds\\
 -\int_0^t \frac{v^\eps_s}{[\lambda^\eps_s(X^\eps_s,Y^\eps_s)]^2}\nabla_s\lambda^\eps_s(X^\eps_s,Y^\eps_s)ds
         \\
         \aad-\int_0^t \frac{v^\eps_s}{[\lambda^\eps_s(X^\eps_s,Y^\eps_s)]^2}[\nabla_X\lambda^\eps_s(X^\eps_s,Y^\eps_s)]^\top p^\eps_sds\\
         \aad-\frac 1{\eps}\int_0^t\frac{v^\eps_s}{[\lambda^\eps_s(X^\eps_s,Y^\eps_s)]^2}[\nabla_Y\lambda^\eps_s(X^\eps_s,Y^\eps_s)]^\top b^\eps_s(X^\eps_s,Y^\eps_s)ds\\
 \aad-\frac 1{2\eps}\int_0^t\frac{v^\eps_s}{[\lambda^\eps_s(X^\eps_s,Y^\eps_s)]^2}\left\|\nabla_{YY}\lambda^\eps_s(X^\eps_s,Y^\eps_s)\right\|_{\Sigma^\eps_s(X^\eps_s,Y^\eps_s)}^2ds\\
       \aad +\frac 1{\eps}\int_0^t\frac{v^\eps_s}{[\lambda^\eps_s(X^\eps_s,Y^\eps_s)]^3}\left\|\nabla_{Y}\lambda^\eps_s(X^\eps_s,Y^\eps_s)\right\|_{\Sigma^\eps_s(X^\eps_s,Y^\eps_s)}^2ds\\
\aad -\frac 1{\sqrt\eps}\int_0^t \frac{v^\eps_s}{[\lambda^\eps_s(X^\eps_s,Y^\eps_s)]^2}[\nabla_Y\lambda^\eps_s(X^\eps_s,Y^\eps_s)]^\top \sigma^\eps_s(X^\eps_s,Y^\eps_s)dW_s\\
%&+\frac 1\eps\int_0^t \frac{|v^\eps_s|^2}{\lambda^4(X^\eps_s,Y^\eps_s)}\left\|\nabla_Y\lambda(X^\eps_s,Y^\eps_s)\right\|_{\Sigma(X^\eps_s,Y^\eps_s)}^2ds\\
%      &+\left(\frac 1{\sqrt\eps}\int_0^t \frac{-v^\eps_s}{[\lambda^\eps_s(X^\eps_s,Y^\eps_s)]^2}[\nabla_Y\lambda(X^\eps_s,Y^\eps_s)]^\top \sigma(X^\eps_s,Y^\eps_s)dw(s)-\frac 1\eps\int_0^t \frac{|v^\eps_s|^2}{\lambda^4(X^\eps_s,Y^\eps_s)}\left\|\nabla_Y\lambda(X^\eps_s,Y^\eps_s)\right\|_{\Sigma(X^\eps_s,Y^\eps_s)}^2ds\right)\\
\aad \      =:\;K^\eps_t+\frac 1{\sqrt\eps}\int_0^t \frac{-v^\eps_s}{[\lambda^\eps_s(X^\eps_s,Y^\eps_s)]^2}[\nabla_Y\lambda^\eps_s(X^\eps_s,Y^\eps_s)]^\top \sigma^\eps_s(X^\eps_s,Y^\eps_s)dW_s,
	\end{aligned}
 \eea
where $K^\eps_t$ is the remaining in the right-hand side.
Therefore, we get
\begin{equation}\label{eq-lnX-2}
\frac1\eps \left[\ln (1+|X^\eps_t|^2)-\ln (1+|x_0^\eps|^2)\right]=\frac 1\eps\hat K^\eps_t+D^\eps_t,
\end{equation}
where
\bea \ad
\hat K^\eps_t:=\Big[K^\eps_t+\frac 1{\eps^2}\int_0^t \frac{|v^\eps_s|^2}{[\lambda^\eps_s(X^\eps_s,Y^\eps_s)]^4}
\left\|\nabla_Y\lambda^\eps_s(X^\eps_s,Y^\eps_s)
\right\|_{\Sigma^\eps_s(X^\eps_s,Y^\eps_s)}^2ds\Big],
\ \hbox{ and }\\
\ad
%\begin{aligned}
D^\eps_t=\frac 1{\eps\sqrt\eps}\int_0^t \frac{-v^\eps_s}{[\lambda^\eps_s(X^\eps_s,Y^\eps_s)]^2}
%&
[\nabla_Y\lambda^\eps_s(X^\eps_s,Y^\eps_s)]^\top \sigma^\eps_s(X^\eps_s,Y^\eps_s)dW_s\\
\aad \qquad \  -\frac 1{\eps^3}\int_0^t \frac{|v^\eps_s|^2}{[\lambda^\eps_s(X^\eps_s,Y^\eps_s)]^4}
\left\|\nabla_Y\lambda^\eps_s(X^\eps_s,Y^\eps_s)
\right\|_{\Sigma^\eps_s(X^\eps_s,Y^\eps_s)}^2ds.
%\end{aligned}
\eea	
Let $\zeta^\eps_L=\inf\left\{t\geq 0:|X^\eps_t|>L\right\}$.
It is obvious that $\zeta^\eps_L$ is an $\F^\eps$-stopping time.
Since $D^\eps_t$ is a local martingale, we have from \eqref{eq-lnX-2} that
\begin{equation}\label{eq-E-eps}
\E^\eps\exp\left\{\frac 1\eps\left[\ln (1+|X^\eps_{t\wedge\zeta^\eps_L}|^2)-\ln (1+|x_0^\eps|^2)-\hat K^\eps_{t\wedge\zeta^\eps_L}\right]\right\}\leq 1.
\end{equation}

On the other hand, from \eqref{formulap} and Assumption \ref{asp-2},  and noting
%that
$\int_0^t e^{-A_\eps(t,s)}ds\leq \int_0^t e^{-\frac{\kappa_0(t-s)}{\eps^2}}ds\leq \frac{\eps^2}{\kappa_0},$ one obtains that there is a finite constant $C_{t,L}$ depending only on $t,L$ satisfies that for all sufficiently small $\eps$,
\begin{equation}\label{eq-p-eps}
|p^\eps_{t\wedge\zeta^\eps_L}|\leq C_{t,L}+x_1^\eps e^{-A_\eps(t)}.
\end{equation}
Similarly, we have for $\eps$ small
\begin{equation}\label{eq-v-eps}
\barray
 |v^\eps_{t\wedge\zeta^\eps_L}|\leq C_{t,L}\int_0^t e^{-\frac{\kappa_0(t-s)}{\eps^2}}ds\leq \eps^2C_{t,L}.\earray
\end{equation}
%Thus,
Combining \eqref{eq-p-eps}, \eqref{eq-v-eps}, the definition of $\hat K^\eps_t$,  Assumption \ref{asp-2}, and
$\limsup_{\eps\to0}\eps x^{\eps}_1<\infty$ %, one
yields that as $\eps$ being small
\begin{equation}\label{eq-B-eps}
\barray
%\begin{aligned}
|\hat K^\eps_{t\wedge\zeta^\eps_L}|\leq C+\eps C_{t,L},
\earray
\end{equation}	
where $C$ is some finite constant
depending neither $\eps$ nor $t, L$.
Therefore, from \eqref{eq-E-eps} and \eqref{eq-B-eps}, we have that for all $\eps$ sufficiently small,
\bea
\E^\eps\exp\Big\{\frac 1\eps\Big[\ln (1+|X^\eps_{t\wedge\zeta^\eps_L}|^2)-\ln (1+|x_0^\eps|^2)-C-\eps C_{t,L}\Big]\Big\}\leq 1.
\eea
Thus, one has for any $t,L,N>0$,
\begin{equation}\label{eq-tight1-1}
%\begin{aligned}
\barray
\PP^\eps \ad
%&
\Big(\sup_{s\in[0,t]}|X^\eps_s|>L\Big)=\PP^\eps(|X^\eps_{t\wedge\zeta^\eps_L}|>L)\\
%&
\ad \leq \PP^\eps(|x_0^\eps|>N)+\E^\eps\exp\left\{\frac 1\eps\left[\ln(1+|X^\eps_{t\wedge\zeta^\eps_L}|^2)-\ln (1+L^2)\right]\right\}\1_{\{|x_0^\eps|\leq N\}}\\
%&
\ad \leq\PP^\eps(|x_0^\eps|>N)+\exp\left\{\frac 1\eps\left[\ln (1+N^2)+C+\eps C_{t,L}-\ln (1+L^2)\right]\right\}.
\earray
%\end{aligned}
\end{equation}
From \eqref{eq-tight1-1} and the logarithm equivalence principle \cite[Lemma 1.2.15]{DZ98}, we obtain that for all $t,N>0$,
\begin{equation}\label{eq-tight1-2}
\lim_{L\to\infty}\limsup_{\eps\to 0}\eps\log \PP\Big(\sup_{s\in[0,t]}|X^\eps_s|>L\Big)\leq\limsup_{\eps\to 0}\eps\log \PP^\eps\Big(|x_0^\eps|>N\Big).
\end{equation}
Because $\{x_0^\eps\}_{\eps>0}$ is exponentially tight and \eqref{eq-tight1-2}, we obtain \eqref{p-cre-11} for any $t>0$.
\end{proof}

\begin{proof}[Proof of \eqref{p-cre-12}]
	By applying Lemma \ref{lm-intbypart} to \eqref{formulap}, one has
	\begin{equation}\label{eq-X-1}
		\begin{aligned}
%	\barray
	X^\eps_t\ad =x_0^\eps+\int_0^tp^\eps_sds=x_0^\eps+\int_0^tx_1^\eps e^{-A_\eps(s)}ds+\dfrac 1{\eps^2}\int_0^t\int_0^se^{-A_\eps(s,r)}F^\eps_s(X^\eps_r,Y^\eps_r)drds
	\\
\ad =x_0^\eps+x_1^\eps\int_0^t e^{-A_\eps(r)}dr+\int_0^t \frac{F^\eps_r(X^\eps_r,Y^\eps_r)}{\lambda^\eps_r(X^\eps_r,Y^\eps_r)}dr\\
\aad-\frac{1}{\lambda^\eps_t(X^\eps_t,Y^\eps_t)}
\int_0^te^{-A_\eps(t,r)}F^\eps_r(X^\eps_r,Y^\eps_r)dr\\
	\aad \ -\int_0^t \frac{1}{[\lambda^\eps_s(X^\eps_s,Y^\eps_s)]^2}\left(\int_0^s e^{-A_\eps(s,r)}F^\eps_r(X^\eps_r,Y^\eps_r)dr\right)d\lambda^\eps_s(X^\eps_s,Y^\eps_s)\\
	\aad \ +\int_0^t \frac{1}{[\lambda^\eps_s(X^\eps_s,Y^\eps_s)]^3}\left(\int_0^s e^{-A_\eps(s,r)}F^\eps_r(X^\eps_r,Y^\eps_r)dr\right)\langle d\lambda^\eps_s(X^\eps_s,Y^\eps_s),d\lambda^\eps_s(X^\eps_s,Y^\eps_s)\rangle_s.
	%	&+\int_0^t \frac{1}{[\lambda^\eps_s(X^\eps_s,Y^\eps_s)]^2}\left(\int_0^s e^{-A_\eps(s,r)}F^\eps_r(X^\eps_r,Y^\eps_r)dr\right)\nabla_Y\lambda(X^\eps_s,Y^\eps_s)dY^\eps_s
%	\earray
\end{aligned}
	\end{equation}
	We obtain from \eqref{eq-X-1} and  It\^o's formula that	
	\begin{equation}\label{eq-X-2}
	\begin{aligned}
%\barray
\ad	X^\eps_t=x_0^\eps+x_1^\eps\int_0^t e^{-A_\eps(r)}dr+\int_0^t \frac{F^\eps_r(X^\eps_r,Y^\eps_r)}{\lambda^\eps_r(X^\eps_r,Y^\eps_r)}dr\\
	\aad-\frac{1}{\lambda^\eps_t(X^\eps_t,Y^\eps_t)}\int_0^te^{-A_\eps(t,r)}F^\eps_r(X^\eps_r,Y^\eps_r)dr\\
	\aad -\int_0^t \frac{\nabla_s\lambda^\eps_s(X^\eps_s,Y^\eps_s) }{[\lambda^\eps_s(X^\eps_s,Y^\eps_s)]^2}\left(\int_0^s e^{-A_\eps(s,r)}F^\eps_r(X^\eps_r,Y^\eps_r)dr\right)ds\\
	\aad -\int_0^t \frac{[\nabla_X\lambda^\eps_s(X^\eps_s,Y^\eps_s)]^\top p^\eps_s }{[\lambda^\eps_s(X^\eps_s,Y^\eps_s)]^2}\left(\int_0^s e^{-A_\eps(s,r)}F^\eps_r(X^\eps_r,Y^\eps_r)dr\right)ds\\
	\aad -\frac 1\eps\int_0^t \frac{[\nabla_Y\lambda^\eps_s(X^\eps_s,Y^\eps_s)]^\top b^\eps_s(X^\eps_s,Y^\eps_s)}{[\lambda^\eps_s(X^\eps_s,Y^\eps_s)]^2}\left(\int_0^s e^{-A_\eps(s,r)}F^\eps_r(X^\eps_r,Y^\eps_r)dr\right)ds\\
	\aad -\frac{1}{2\eps}\int_0^t\frac{\left\|\nabla_{YY}\lambda^\eps_s(X^\eps_s,Y^\eps_s)\right\|_{\Sigma^\eps_s(X^\eps_s,Y^\eps_s)}^2}{[\lambda^\eps_s(X^\eps_s,Y^\eps_s)]^2}\left(\int_0^s e^{-A_\eps(s,r)}F^\eps_r(X^\eps_r,Y^\eps_r)dr\right)ds\\
	\aad +\frac{1}{\eps}\int_0^t\frac{\left\|\nabla_{Y}\lambda^\eps_s(X^\eps_s,Y^\eps_s)\right\|_{\Sigma^\eps_s(X^\eps_s,Y^\eps_s)}^2}{[\lambda^\eps_s(X^\eps_s,Y^\eps_s)]^3}\left(\int_0^s e^{-A_\eps(s,r)}F^\eps_r(X^\eps_r,Y^\eps_r)dr\right)ds\\
	&-\frac 1{\sqrt\eps}\int_0^t \frac{\big(\int_0^s e^{-A_\eps(s,r)}F^\eps_r(X^\eps_r,Y^\eps_r)dr\big)}{[\lambda^\eps_s(X^\eps_s,Y^\eps_s)]^2}\big([\nabla_Y\lambda^\eps_s(X^\eps_s,Y^\eps_s)]^\top \sigma^\eps_s(X^\eps_s,Y^\eps_s)dW_s\big)\\
\ad 	=:\bar K^\eps_t-\bar D^\eps_t,
%	\earray
\end{aligned}
	\end{equation}
	where
	$$
	\bar D^\eps_t:=\frac 1{\sqrt\eps}\int_0^t \frac{\left(\int_0^s e^{-A_\eps(s,r)}F^\eps_r(X^\eps_r,Y^\eps_r)dr\right)}{[\lambda^\eps_s(X^\eps_s,Y^\eps_s)]^2}\left([\nabla_Y\lambda^\eps_s(X^\eps_s,Y^\eps_s)]^\top \sigma^\eps_s(X^\eps_s,Y^\eps_s)dW_s\right),
	$$
	and $\bar K^\eps_t$ is the remaining in the right-hand side of \eqref{eq-X-2}.
	By the regularity of $\lambda^\eps_t$,
it is not difficult to see that
	\begin{equation}\label{eq-lambda-1}
	\begin{aligned}
 &\Big|\frac{1}{\lambda^\eps_t(X^\eps_t,Y^\eps_t)}\int_0^te^{-A_\eps(t,r)}F^\eps_r(X^\eps_r,Y^\eps_r)ds-\frac{1}{\lambda^\eps_s(X^\eps_s,Y^\eps_s)}\int_0^se^{-A_\eps(s,r)}F^\eps_r(X^\eps_r,Y^\eps_r)ds\Big|\\
	&\leq \Big|\frac{e^{-A_\eps(t)}}{\lambda^\eps_t(X^\eps_t,Y^\eps_t)}-\frac{e^{-A_\eps(s)}}{\lambda^\eps_s(X^\eps_s,Y^\eps_s)}\Big|\int_0^s e^{A_\eps(r)}|F^\eps_r(X^\eps_r,Y^\eps_r)|dr\\
	&\quad+\frac{e^{-A_\eps(t)}}{\lambda^\eps_t(X^\eps_t,Y^\eps_t)}\int_s^te^{A_\eps(r)}|F^\eps_r(X^\eps_r,Y^\eps_r)|dr\\
	&\leq  C\frac{|t-s|}{\eps^2}\sup_{r\in[s,t]}|\lambda^\eps_r(X^\eps_r,Y^\eps_r)|\int_0^se^{-\frac{\kappa_0}{\eps^2}(s-r)}|F^\eps_r(X^\eps_r,Y^\eps_r)|dr\\
	&\quad+C\int_s^te^{-\frac{\kappa_0}{\eps^2}(t-r)} |F^\eps_r(X^\eps_r,Y^\eps_r)|dr.
	\end{aligned}
	\end{equation}
	We obtain from definition of $\bar K^\eps_t$,
%on
an application of \eqref{eq-lambda-1},
and
recalling definition of $\zeta^\eps_L$
that there is a finite constant $C_L$ depending only on $L$ such that for all small $\eps$,
	\begin{equation}
	\sup_{s\in[0,T]}\sup_{t\in[s,s+\delta]}|\bar K^\eps_{t\wedge\zeta^\eps_L}-\bar K^\eps_{s\wedge\zeta^\eps_L}|\leq C_{T,L}\delta,\text{ for all }T>0, 0<\delta<1.
	\end{equation}
	
	Now, let $T>0, \ell>0$ be fixed, and
%Let
$L>0$ be fixed but otherwise arbitrary.
% but fixed at the moment.
	We have that for any small $\delta$ satisfying $\delta<1$, $C_{T,L}\delta<\ell/2$ and small $\eps$
	\begin{equation}\label{eq-Xt-Xs}
	\barray
	\PP^\eps \ad\Big(\sup_{t\in[s,s+\delta]}|X^\eps_t-X^\eps_s|> \ell\Big)\\
	\ad\leq \PP^\eps(\zeta^\eps_L\leq T+1)+\PP^\eps \Big(\sup_{t\in[s,s+\delta]}|X^\eps_{t\wedge\zeta^\eps_L}-X^\eps_{s\wedge\zeta^\eps_L}|> \ell\Big)\\
	\ad\leq \PP^\eps(\zeta^\eps_L\leq T+1)+\PP^\eps \Big(\sup_{t\in[s,s+\delta]}|\bar D^\eps_{t\wedge\zeta^\eps_L}-\bar D^\eps_{s\wedge\zeta^\eps_L}|> \frac{\ell}2\Big)\\
	\ad\leq \PP^\eps(\sup_{t\in[0,T+1]}|X^\eps_t|>L)+\sum_{k=1}^d\PP^\eps \Big(\sup_{t\in[s,s+\delta]}|\bar D^{\eps,k}_{t\wedge\zeta^\eps_L}-\bar D^{\eps,k}_{s\wedge\zeta^\eps_L}|> \frac{\ell}2\Big),
	\earray\end{equation}
    where $\bar D^{\eps,k}_t$ is the $k$-th component of $\bar D^\eps_t$, $k=1,\dots,d$. It is readily seen that $\{\bar D^{\eps,k}_{t\wedge\zeta^\eps_L}-\bar D^{\eps,k}_{s\wedge\zeta^\eps_L}\}_{t\geq s}$ is a martingale with the quadratic variations bounded by
    $$
    \frac {C_L}{\eps}\int_{s\wedge\zeta^\eps_L}^{t\wedge\zeta^\eps_L}\int_0^s e^{-\frac{2\kappa_0(s-r)}{\eps^2}}drds\leq \eps C_L\delta.
    $$
    By the exponential
     martingale inequality \cite[Theorem 7.4,
    p. 44]{Mao97}, we have
   \begin{equation}\label{eq-Dt-Ds}
   \barray   \PP^\eps \ad \Big(\sup_{t\in[s,s+\delta]}|\bar D^{\eps,k}_{t\wedge\zeta^\eps_L}-\bar D^{\eps,k}_{s\wedge\zeta^\eps_L}|> \frac{\ell}2\Big)
   \\ \ad\leq \PP^\eps \Big(\sup_{t\in[s,s+\delta]}|\bar D^{\eps,k}_{t\wedge\zeta^\eps_L}-\bar D^{\eps,k}_{s\wedge\zeta^\eps_L}|> \frac{\ell}4+\frac{\ell}{4\eps C_L\delta}\eps C_L\delta\Big)\\
   \ad\leq \exp\Big\{-\frac{\ell^2}{8\eps C_L\delta}\Big\}.
   \earray
   \end{equation}
   Combining \eqref{eq-Xt-Xs} and \eqref{eq-Dt-Ds}, the logarithm equivalence principle \cite[Lemma 1.2.15]{DZ98} yields that
   \begin{equation}
   	\begin{aligned}
   \lim_{\delta\to0}&\limsup_{\eps\to0}\sup_{s\in[0,T]}\eps\log\PP^\eps \Big(\sup_{t\in[s,s+\delta]}|X^\eps_t-X^\eps_s|> \ell\Big)\\
   &\leq\limsup_{\eps\to0} \eps\log\PP^\eps(\sup_{t\in[0,T+1]}|X^\eps_t|>L),\; \forall L>0.
   \end{aligned}
   \end{equation}
   Letting $L\to\infty$ and using \eqref{p-cre-11}, we obtain \eqref{p-cre-12}.
\end{proof}

\begin{proof}[Proof of \eqref{p-cre-21} and \eqref{p-cre-22}]
Once we established the exponential tightness of $\{X^\eps\}_{\eps>0}$, the proof of \eqref{p-cre-21} and \eqref{p-cre-22} for $\{\mu^\eps\}_{\eps>0}$, which is in fact the occupation measure of a diffusion,
 is similar to that of the
 first-order coupled systems. As a consequence, such proofs can be found in \cite[p. 3134]{Puh16}.
\end{proof}

\subsubsection{Characterization of
Rate Function}\label{sec:iden}
Let $\beta(s)\in \C(\R_+,\R^d)$ be a step function satisfying that there are $0=t_0<t_1<\dots<t_m<\infty$ and $\beta_i\in\R, i=1,\dots,m$ such that
\begin{equation}\label{eq-beta}\barray
\beta(s)=\sum_{i=1}^m \beta_i\1_{[t_{i-1},t_{i})}(s).\earray
\end{equation}
For $\varphi_s\in\C(\R_+,\R^d)$ and $\beta(s)$ of the form \eqref{eq-beta}, we define
\begin{equation}\label{eq-betadx}\barray
\int_0^t \beta(s)d\varphi_s:=\sum_{i=1}^{m}\beta_i^\top(\varphi_{t\wedge t_i}-\varphi_{t\wedge t_{i-1}}).\earray
\end{equation}
Now, for each step function $\beta(s)$, each $f(t,x,y)$
%representing a
real-valued $\C^{1,2,2}(\R_+\times \R^d \times \R^l)$-function with compact support in $y$ locally uniformly in $(t, x)$, and
each $(\varphi,\mu)\in \C(\R_+,\R^d)\times \C_{\uparrow}(\R_+,\M(\R^l))$,
let
\begin{equation}\label{eq-phi-1}
%\barray
\begin{aligned}
\Phi_t^{\beta,f}(\varphi,\mu):=\ad\int_0^t \beta(s) d\varphi_s-\int_0^t \int_{\R^l} \frac{[\beta(s)]^\top F_s(\varphi_s,y)}{\lambda_s(\varphi_s,y)}\mu(ds,dy)\\
\ad \ -\int_0^t\int_{\R^l} [\nabla_y f(s,\varphi_s,y)]^\top b_s(\varphi_s,y)\mu(ds,dy)\\
\ad \ -\frac 12\int_0^t\int_{\R^l} \tr \Big(\Sigma_s(\varphi_s,y)\nabla_{yy}f(s,\varphi_s,y)\Big)\mu(ds,dy)\\
\ad -\frac 12 \int_0^t\int_{\R^l} \|\nabla_yf(s,\varphi_s,y)\|^2_{\Sigma_s(\varphi_s,y)}\mu(ds,dy).
\end{aligned}
%\earray
\end{equation}
Moreover, let $\tau(\varphi,\mu)$ be
% represent
a continuous function of $(\varphi,\mu)\in\C(\R_+,\R^d) \times \C_{\uparrow}(\R_+,\M(\R^l))$ that is also a stopping time relative to the flow
$G = (G_t,t\in\R_+)$ on
$\C(\R_+,\R^d)\times \C_{\uparrow}(\R_+,\M(\R^l))$
of the $\sigma$-algebra
 $G_t$
 %, where $G_t$ is
 generated by the mappings $\varphi \to \varphi_s$ and $\mu \to \mu_s$ for $s \leq t.$
%(We note that the flow G is not right continuous, so τ is a strict stopping time; see Jacod and Shiryaev [25].)
Let us also assume that $\varphi_{t\wedge\tau(\varphi,\mu)}$ is a bounded function of $(\varphi,\mu)$.
It is  seen that under Assumption \ref{asp-3}, $\Phi^{\beta,f}_t(\varphi,\mu)$ is continuous in $(\varphi,\mu)$.

Next, define
\begin{equation}\label{I*}
\I^*(\varphi,\mu)=\sup_{\beta,f,t,\tau} \Phi^{\beta,f}_{t\wedge\tau(\varphi,\mu)}(\varphi,\mu),
\end{equation}
where the supremum is taken over $\beta(s)$, $f(s,x,y)$, and $\tau(\varphi,\mu)$ satisfying the
requirements as the above and over $t\geq 0$.
It is seen that $\I^*$ is lower semi-continuous in $(\varphi,\mu)$.

Now, let $\hat\I$ be a large deviations limit rate functions or (large deviations) LD limit points of
$\{(X^\eps,\mu^\eps)\}_{\eps>0}$ (i.e., a rate function of some subsequence of $\{(X^\eps,\mu^\eps)\}_{\eps>0}$ that obeys the LDP)
such that $\hat\I(\varphi,\mu) = \infty$ unless $\varphi_0 = \hat x$, where $\hat x$ is a preselected element of $\R^d$.
This restriction will be removed in Section \ref{sec:com}.
We will identify the rate functions.
For any such a large deviation limit point $\hat \I$, we aim to prove $\hat\I=\I^*$
 %at sufficiently regular points
 by showing
 the upper bound
  $\hat\I\geq \I^*$
%(which in fact is true at any points)
and
the lower bound
$\hat\I\leq \I^*$;
see detail in Section \ref{sec:IsIhat} and Section \ref{sec:iden-com}.
Moreover, it %also
will be seen that $\I^*(\varphi,\mu) =
\I(\varphi,\mu)$ provided
$\I^*(\varphi,\mu) < \infty$, $\varphi_0 =\hat x$, and $\I_0(\hat x)= 0.$
Throughout this section, the assumptions in Theorem \ref{thm-main} are always assumed to be satisfied.

\begin{rem}\label{rem-25}{\rm
%	It is noted
Note that $\I^*$ is defined
similarly
but not identical
%the same
as that in the case of first-order coupled systems in \cite{Puh16} although the solution of \eqref{eq:F-setup}
%will
shares the same rate function with the corresponding first-order system.
%	Precisely,
Compared with \cite{Puh16}, $\I^*$ is defined by taking the supremum over smaller space when we did not allow $\beta$ to be a function of $X$.
	This modification has an important role in the proof of the lower bound of the LD limits, i.e., the inequality $\I^*\leq \hat \I$.
	Otherwise, it would be
	%very difficult
	impossible
	to control terms containing the derivative $p^\eps$ of $X^\eps$, specially the integral involving $p^\eps$ and the diffusion part of the fast process (see \eqref{eq-betadx-1}).
	Meanwhile, it would have led to a difficulty in proving the upper bound of the LD limits, i.e., the inequality $\I^*\geq \hat \I$.
	%However,
	%since we are handling the problem without the diffusion part in $X^\eps$,
However, it will be shown that we still can get the upper bound, as in the first-order system (in \cite{Puh16});
see the details in Section \ref{sec:iden-com}.
}\end{rem}

\subsubsection{Lower Bound of Large Deviations Limits}\label{sec:IsIhat}
This section is devote to proving $\I^*\leq \hat \I$.
We have the following theorem.

\begin{thm}\label{thm-hatI}
	Let $\hat \I$ be a LD limit point of $\{(X^\eps,\mu^\eps)\}_{\eps>0}$. For any $t>0$, $\beta,f,\tau$ are as given
above,
	\begin{equation}\label{eq-phi-I}\barray
	\sup_{(\varphi,\mu)\in\C(\R_+,\R^d) \times \C_{\uparrow}(\R_+,
		\M(\R^l))}\Big(\Phi^{\beta,f}_{t\wedge\tau(\varphi,\mu)}(\varphi,\mu)-\hat \I(\varphi,\mu)\Big)=0. \earray
	\end{equation}
Then $\I^*(\varphi,\mu)\leq\hat\I(\varphi,\mu)$ for all $(\varphi,\mu)\in \C(\R_+,\R^d) \times \C_{\uparrow}(\R_+,
	\M(\R^l))$.
\end{thm}

\begin{proof}
	For $\beta(s)$ being of the form \eqref{eq-beta} and $f(s,x,y)$ being a function with compact support in $y$ locally uniformly in $(t,x)$, denote
\begin{equation}\label{eq-phi-3}
\begin{aligned}
\ad\Gamma^{\eps,\beta}_t(\varphi,\mu)\\
\ad  =-\int_0^t\int_{\R^l} [\beta(s)]^\top x_1^\eps e^{-A^{\varphi,y}_\eps(s)}\mu(ds,dy)\\
\aad+\int_0^t\int_{\R^l}\frac{\beta(t)}{\lambda^\eps_t(\varphi_t,y)}e^{-A_\eps^{\varphi,y}(t,s)}F^\eps_s(\varphi_s,y)\mu(ds,dy)\\
\aad  +\int_0^t\int_{\R^l} [\beta(s)]^\top\left(\int_0^s e^{-A^{\varphi,y}_\eps(s,r)}F^\eps_r(\varphi_r,y)dr\right)\frac{\nabla_s\lambda^\eps_s(\varphi_s,y) }{[\lambda^\eps_s(\varphi_s,y)]^2}\mu(ds,dy)\\
\aad +\int_0^t \int_{\R^l}[\beta(s)]^\top\left(\int_0^s e^{-A^{\varphi,y}_\eps(s,r)}F^\eps_r(\varphi_r,y)dr\right)\frac{[\nabla_X\lambda^\eps_s(\varphi_s,y)]^\top \dot\varphi_s }{[\lambda^\eps_s(\varphi_s,y)]^2}\mu(ds,dy)\\
\aad+\frac 1\eps\int_0^t\int_{\R^l}[\beta(s)]^\top\left(\int_0^s e^{-A^{\varphi,y}_\eps(s,r)}F^\eps_r(\varphi_r,y)dr\right)\frac{[\nabla_Y\lambda^\eps_s(\varphi_s,y)]^\top b^\eps_s(\varphi_s,y)}{[\lambda^\eps_s(\varphi_s,y)]^2}\mu(ds,dy)\\
\aad+\frac{1}{2\eps}\int_0^t \int_{\R^l}[\beta(s)]^\top\left(\int_0^s e^{-A^{\varphi,y}_\eps(s,r)}F^\eps_r(\varphi_r,y)dr\right)\frac{\left\|\nabla_{YY}\lambda^\eps_s(\varphi_s,y)\right\|_{\Sigma^\eps_s(\varphi_s,y)}^2}{[\lambda^\eps_s(\varphi_s,y)]^2}\mu(ds,dy)\\
\aad-\frac{1}{\eps}\int_0^t\int_{\R^l} [\beta(s)]^\top\left(\int_0^s e^{-A^{\varphi,y}_\eps(s,r)}F^\eps_r(\varphi_r,y)dr\right)\frac{\left\|\nabla_{Y}\lambda^\eps_s(\varphi_s,y)\right\|_{\Sigma^\eps_s(\varphi_s,y)}^2}{[\lambda^\eps_s(\varphi_s,y)]^3}\mu(ds,dy)\\
\aad-\frac 1{2\eps^2}\int_0^t\int_{\R^l} \frac{\Big|[\beta(s)]^\top\left(\int_0^s e^{-A^{\varphi,y}_\eps(s,r)}F^\eps_r(\varphi_r,y)dr\right)\Big|^2}{[\lambda^\eps_s(\varphi_s,y)]^4}\left\|\nabla_Y\lambda^\eps_s(\varphi_s,y)\right\|_{\Sigma^\eps_s(\varphi_s,y)}^2\mu(ds,dy),\\
\aad +\frac1{\eps}\int_0^t\int_{\R^l}\frac{[\beta(s)]^\top\left(\int_0^s e^{-A^{\varphi,y}_\eps(s,r)}F^\eps_r(\varphi_r,y)dr\right)}{[\lambda^\eps_s(\varphi_s,y)]^2}\left\|\nabla_Y\lambda^\eps_s(\varphi_s,y)\right\|_{\Sigma^\eps_s(\varphi_s,y)}^2\mu(ds,dy),
\end{aligned}
\end{equation}
where
$
A^{\varphi,y}_\eps(t,s):=\frac 1{\eps^2}\int_s^t\int_{\R^l}\lambda^\eps_r(\varphi_r,y)\mu(dr,dy),
$
and
\begin{equation}\label{eq-phi-4}
\barray
\Psi_t^{\eps,\beta,f}(\varphi,\mu)\ad := f(t,\varphi_t,Y^\eps_t)-f(0,\varphi_0,y_0^\eps)-\int_0^t \int_{\R^l} \nabla_sf(s,\varphi_s,y)\mu(ds,dy)\\
\aad -\int_0^t\int_{\R^l} [\nabla_x f(s,\varphi_s,y)]^\top \dot\varphi_s\mu(ds,dy),
\earray
\end{equation}
and
\begin{equation}\label{eq-phi-2}
\barray
\Phi_t^{\eps,\beta,f}(\varphi,\mu):=\ad\int_0^t \beta(s) d\varphi_s-\int_0^t \int_{\R^l} \frac{[\beta(s)]^\top F^\eps_s(\varphi_s,y)}{\lambda^\eps_s(\varphi_s,y)}\mu(ds,dy)\\
\aad -\int_0^t\int_{\R^l} [\nabla_y f(s,\varphi_s,y)]^\top b^\eps_s(\varphi_s,y)\mu(ds,dy)\\
\aad-\frac 12\int_0^t\int_{\R^l} \tr \Big(\Sigma^\eps_s(\varphi_s,y)\nabla_{yy}f(s,\varphi_s,y)\Big)\mu(ds,dy)\\
\aad -\frac 12 \int_0^t\int_{\R^l} \|\nabla_yf(s,\varphi_s,y)\|^2_{\Sigma^\eps_s(\varphi_s,y)}\mu(ds,dy).
\earray
\end{equation}
We have from \eqref{formulap} and Lemma \ref{lm-intbypart} that
\begin{equation}\label{eq-betadx-1}
\barray
\ad\int_0^t [\beta(s)]^\top dX^\eps_s=\int_0^t[\beta(s)]^\top p^\eps_sds\\
\ad =\int_0^t [\beta(s)]^\top x_1^\eps e^{-A_\eps(s)}ds+\frac 1{\eps^2}\int_0^t [\beta(s)]^\top\int_0^s e^{-A_\eps(s,r)}F^\eps_r(X^\eps_r,Y^\eps_r)dr\\
\ad =\int_0^t [\beta(s)]^\top x_1^\eps e^{-A_\eps(s)}ds+\int_0^t\frac{[\beta(s)]^\top F^\eps_s(X^\eps_s,Y^\eps_s)}{\lambda^\eps_s(X^\eps_s,Y^\eps_s)}ds\\
\aad \ -\frac{\beta(t)}{\lambda^\eps_t(X^\eps_t,Y^\eps_t)}\int_0^te^{-A_\eps(t,s)}F^\eps_s(X^\eps_s,Y^\eps_s)ds\\
\aad \ -\int_0^t[\beta(s)]^\top\left(\int_0^s e^{-A_\eps(s,r)}F^\eps_r(X^\eps_r,Y^\eps_r)dr\right)\frac{\nabla_s\lambda^\eps_s(X^\eps_s,Y^\eps_s) }{[\lambda^\eps_s(X^\eps_s,Y^\eps_s)]^2}ds\\
\aad \ -\int_0^t [\beta(s)]^\top\left(\int_0^s e^{-A_\eps(s,r)}F^\eps_r(X^\eps_r,Y^\eps_r)dr\right)\frac{[\nabla_X\lambda^\eps_s(X^\eps_s,Y^\eps_s)]^\top p^\eps_s }{[\lambda^\eps_s(X^\eps_s,Y^\eps_s)]^2}ds\\
\aad \ -\frac 1\eps\int_0^t[\beta(s)]^\top\left(\int_0^s e^{-A_\eps(s,r)}F^\eps_r(X^\eps_r,Y^\eps_r)dr\right)\frac{[\nabla_Y\lambda^\eps_s(X^\eps_s,Y^\eps_s)]^\top b^\eps_s(X^\eps_s,Y^\eps_s)}{[\lambda^\eps_s(X^\eps_s,Y^\eps_s)]^2}ds\\
\aad \ -\frac{1}{2\eps}\int_0^t [\beta(s)]^\top\left(\int_0^s e^{-A_\eps(s,r)}F^\eps_r(X^\eps_r,Y^\eps_r)dr\right)\frac{\left\|\nabla_{YY}\lambda^\eps_s(X^\eps_s,Y^\eps_s)\right\|_{\Sigma^\eps_s(X^\eps_s,Y^\eps_s)}^2}{[\lambda^\eps_s(X^\eps_s,Y^\eps_s)]^2}ds\\
\aad \ +\frac{1}{\eps}\int_0^t [\beta(s)]^\top\left(\int_0^s e^{-A_\eps(s,r)}F^\eps_r(X^\eps_r,Y^\eps_r)dr\right)\frac{\left\|\nabla_{Y}\lambda^\eps_s(X^\eps_s,Y^\eps_s)\right\|_{\Sigma^\eps_s(X^\eps_s,Y^\eps_s)}^2}{[\lambda^\eps_s(X^\eps_s,Y^\eps_s)]^3}ds\\
\aad \ -\frac 1{\sqrt\eps}\int_0^t\frac{[\beta(s)]^\top\left(\int_0^s e^{-A_\eps(s,r)}F^\eps_r(X^\eps_r,Y^\eps_r)dr\right)}{[\lambda^\eps_s(X^\eps_s,Y^\eps_s)]^2}[\nabla_Y\lambda^\eps_s(X^\eps_s,Y^\eps_s)]^\top \sigma^\eps_s(X^\eps_s,Y^\eps_s)dW_s.
\earray
\end{equation}
Moreover, It\^o's formula yields that
\begin{equation}\label{eq-itof}
\barray
\ad f(t,X^\eps_t,Y^\eps_t)-f(0,x_0^\eps,y_0^\eps)=\int_0^t \nabla_s f(s,X^\eps_s,Y^\eps_s)ds+\int_0^t [\nabla_Xf(s,X^\eps_s,Y^\eps_s)]^\top p^\eps_sds\\
\aad \ +\frac 1\eps\int_0^t[\nabla_Yf(s,X^\eps_s,Y^\eps_s)]^\top b^\eps_s(X^\eps_s,Y^\eps_s)ds+\frac 1{\sqrt\eps}[\nabla_Yf(s,X^\eps_s,Y^\eps_s)]^\top\sigma^\eps_s(X^\eps_s,Y^\eps_s)dW_s\\
\aad \ +\frac 1{2\eps}\int_0^t\tr\Big(\Sigma^\eps_s(X^\eps_s,Y^\eps_s)\nabla_{YY}f(s,X^\eps_s,Y^\eps_s)\Big)ds.
\earray
\end{equation}

Combining \eqref{eq-itof},  \eqref{eq-betadx-1},  the definition of $\Phi^{\eps,\beta,f}_t$, $\Gamma^{\eps,\beta}_t$, $\Psi^{\eps,\beta,f}_t$ in \eqref{eq-phi-2},  \eqref{eq-phi-3}, and \eqref{eq-phi-4}, we obtain that
\begin{equation}\label{eq-1eps-mar}
\barray
\ad\frac1\eps \Big[\Phi^{\eps,\beta,f}_t(X^\eps,\mu^\eps)
+\Gamma^{\eps,\beta}_t(X^\eps,\mu^\eps)\Big]+\Psi^{\eps,\beta,f}_t(X^\eps,\mu^\eps)\\
\ad=\frac 1{\sqrt\eps}\int_0^t[\nabla_Yf(s,X^\eps_s,Y^\eps_s)]^\top\sigma^\eps_s(X^\eps_s,Y^\eps_s)dW_s\\
\aad-\frac 1{2\eps} \int_0^t \|\nabla_Yf(s,X^\eps_s,Y^\eps_s)\|^2_{\Sigma^\eps_s(X^\eps_s,Y^\eps)}ds
\\
\aad-\frac 1{\eps\sqrt\eps}\int_0^t\frac{[\beta(s)]^\top\left(\int_0^s e^{-A_\eps(s,r)}F^\eps_r(X^\eps_r,Y^\eps_r)dr\right)}{[\lambda^\eps_s(X^\eps_s,Y^\eps_s)]^2}[\nabla_Y\lambda^\eps_s(X^\eps_s,Y^\eps_s)]^\top \sigma^\eps_s(X^\eps_s,Y^\eps_s)dW_s\\
\aad -\frac 1{2\eps^3}\int_0^t \frac{\big|[\beta(s)]^\top\left(\int_0^s e^{-A_\eps(s,r)}F^\eps_r(X^\eps_r,Y^\eps_r)dr\right)\big|^2}{[\lambda^\eps_s(X^\eps_s,Y^\eps_s)]^4}\left\|\nabla_Y\lambda^\eps_s(X^\eps_s,Y^\eps_s)\right\|_{\Sigma^\eps_s(X^\eps_s,Y^\eps_s)}^2ds\\
\aad +\frac1{\eps^2}\int_0^t\frac{[\beta(s)]^\top\left(\int_0^s e^{-A_\eps(s,r)}F^\eps_r(X^\eps_r,Y^\eps_r)dr\right)}{[\lambda^\eps_s(X^\eps_s,Y^\eps_s)]^2}\left\|\nabla_Y\lambda^\eps_s(X^\eps_s,Y^\eps_s)\right\|_{\Sigma^\eps_s(X^\eps_s,Y^\eps_s)}^2ds.
\earray
\end{equation}
Since the right-hand side of \eqref{eq-1eps-mar} is a local martingale and $\tau(X^\eps,\mu^\eps)$ is a stopping time with respect to $\F^\eps$ due to the measurability of $X^\eps_t,\mu^\eps_t$ with respect to $\F_t^\eps$, we have that
\begin{equation}\label{eq-conv}
\barray
\E^\eps\exp\Big\{\frac 1\eps\Big[\Phi^{\eps,\beta,f}_{t\wedge \tau(X^\eps,\mu^\eps)}(X^\eps,\mu^\eps)\!+\!\Gamma^{\eps,\beta}_{t\wedge \tau(X^\eps,\mu^\eps)}(X^\eps,\mu^\eps)\!+\!\eps\Psi^{\eps,\beta,f}_{t\wedge \tau(X^\eps,\mu^\eps)}(X^\eps,\mu^\eps)\Big]\Big\}=1.
\earray
\end{equation}

%To proceed, we restate the following lemma, which can be found in
%\cite[Theorem 3.3]{Puh16} or
%\cite[Theorem 2.1.10]{DS89}.

\begin{lm}\label{lm-vah}$($\cite[Theorem 2.1.10]{DS89}$)$
	Assume that the net $\{\nu_\eps\}_{\eps>0}$ is  exponentially tight and let $\hat \I$ represent an LD limit point of $\{\nu_\eps\}_{\eps>0}$. Let $\Phi_\eps$ be a net of uniformly bounded real-valued functions on $\mathbb S$ such that $\int_{\mathbb S}\exp(\frac 1\eps \Phi_\eps(z))\nu_\eps(dz) = 1.$
	If $\Phi_\eps$ converges to $\Phi$ uniformly on compact sets (as $\eps\to0$)
with the function $\Phi$ being
continuous, then $\sup_{z\in\mathbb S}(\Phi(z) - \hat \I(z)) = 0.$
 \end{lm}

As in the proof of \eqref{p-cre-11} and \eqref{p-cre-12}  in Section \ref{sec:exp}, it is not difficult to obtain from  Assumption \ref{asp-2} and the fact $\varphi_{t\wedge\tau(\varphi,\mu)}$ is bounded function of $(\varphi,\mu)$ that there is finite a constant $C$, which is independent of $\eps$ such that for all small enough $\eps$
$
|\Gamma^{\eps,\beta}_{t\wedge \tau(\varphi,\mu)}(\varphi,\mu)|\leq C\eps\text{ uniformly over }(\varphi,\mu).
$
Similarly, there is a constant $C$ such that for  $\eps$  sufficiently small,
$
|\Psi^{\eps,\beta,f}_{t\wedge \tau(\varphi,\mu)}(\varphi,\mu)|<C\text{ uniformly over }(\varphi,\mu).
$
As a result, one has
\begin{equation}\label{eq-conv-1}
\Gamma^{\eps,\beta}_{t\wedge \tau(\varphi,\mu)}(\varphi,\mu)+\eps\Psi^{\eps,\beta,f}_{t\wedge \tau(\varphi,\mu)}(\varphi,\mu)\to 0
\end{equation}
as $\eps\to0$ uniformly in compact sets.
Finally, by assumption \eqref{eq-asp-conv}, we have
\begin{equation}\label{eq-conv-3}
	\Phi^{\eps,\beta,f}_{t\wedge \tau(\varphi,\mu)}(\varphi,\mu)\to\Phi^{\beta,f}_{t\wedge \tau(\varphi,\mu)}(\varphi,\mu)
\end{equation}
as $\eps\to0$ uniformly in compact sets.
Combining \eqref{eq-conv} and \eqref{eq-conv-1} and then applying Lemma \ref{lm-vah} yields \eqref{eq-phi-I}.
Then, it follows immediately that $\I^*(\varphi,\mu)\leq\hat\I(\varphi,\mu)$ for all $(\varphi,\mu)\in \C(\R_+,\R^d) \times \C_{\uparrow}(\R_+,
\M(\R^l))$.
The proof is complete.
\end{proof}

\subsubsection{Upper Bound of Large Deviations Limits}\label{sec:iden-com}
Let $\hat\I$ be a large deviations limit point of
$\{(X^\eps,\mu^\eps)\}_{\eps>0}$
such that $\hat\I(\varphi,\mu) = \infty$ unless $\varphi_0 = \hat x$,
a preselected element of $\R^d$.
In this section, we aim to prove that
$\hat \I(\varphi,\mu)\leq \I^*(\varphi,\mu)$, for any $(\varphi,\mu)\in \C(\R_+,\R^d) \times \C_{\uparrow}(\R_+,
\M(\R^l))$ such that $\varphi_0=\hat x$.
The completion of the proof will be given later in Section \ref{sec:com}.
With the results established in Sections \ref{sec:exp} and \ref{sec:IsIhat},
this part can be done similarly to that of \cite[Sections 6-8]{Puh16}  because  the rate function has a similar variational representation.
Although our $\I^*$ is
defined as the supremum in a smaller space than in \cite{Puh16}, we can still prove $\hat \I\leq \I^*$ by a similar argument as in \cite{Puh16}.
We will only provide a sketch of the main ideas  and highlight the differences, whereas
detailed arguments will be referred to
\cite[Sections 6-8]{Puh16}.

It is obvious that it suffices to consider the case $\I^*(\varphi,\mu)<\infty.$
Therefore, we should investigate the regularity of $(\varphi,\mu)$ provided $\I^*(\varphi,\mu)<\infty$ first.
It is shown in \cite[Section 6]{Puh16} that if $(\varphi,\mu)\in \C(\R_+,\R^d)\times\C_{\uparrow}(\R_+,\R^l)$, $\I^*(\varphi,\mu)<\infty$ then
$\mu(ds,dy)=m_s(y)dyds$ and
$\varphi_s$ is absolutely continuous (w.r.t Lebesgue measure on $\R_+$), $m_s(y)$ is a probability density function in $\R^l$.
% belonging to $\mathcal P(\R^l)$.
In this case, $\I^*$ has the following representation
\begin{equation}\label{eq-I*}
	\begin{aligned}
		&\I^*(\varphi,\mu)\\
		=&\int_0^\infty \Bigg(\sup_{\beta\in\R^d}\bigg(\beta^\top \dot\varphi_s-\beta^\top\int_{\R^l} \frac{ F_s(\varphi_s,y)m_s(y)}{\lambda_s(\varphi_s,y)}dy\bigg)\\
		&+\sup_{h\in\C_0^1(\R^l)}\int_{\R^l} [\nabla_y h(y)]^\top \bigg(\frac12\di(\Sigma_s(\varphi_s,y)m_s(y))-b_s(\varphi_s,y)m_s(y)\bigg)dy\\
		%&-\frac 12\int_0^t\int_{\R^l} \tr \Big(\Sigma_s(\varphi_s,y)\nabla_{yy}f(s,\varphi_s,y)\Big)\mu(ds,dy)\\
		&-\frac 12 \int_{\R^l} \|\nabla_yh(y)\|^2_{\Sigma_s(\varphi_s,y)}m_s(y)dy\Bigg)ds\\
		=&\int_0^\infty \Bigg(\sup_{\beta\in\R^d}\bigg(\beta^\top \dot\varphi_s-\beta^\top\int_{\R^l} \frac{ F_s(\varphi_s,y)m_s(y)}{\lambda_s(\varphi_s,y)}dy\bigg)\\
		&+\sup_{g\in\mathbb L_0^{1,2}(\R^l,\R^l,\Sigma_s(\varphi_s,y),
			m_s(y) dy)}\int_{\R^l}\bigg([g(y)]^\top \Sigma_s(\varphi_s,y)\bigg(\frac{\nabla_ym_s(y)}{2m_s(y)}-\J_{s,m_s(\cdot),\varphi_s}(y)\bigg)\\
		&\hspace{2.5cm}-\frac 12\|g(y)\|_{\Sigma_s(\varphi_s,y)}\bigg)m_s(y)\bigg)dy\Bigg)ds.
	\end{aligned}
\end{equation}
In the above, $\mathbb L_0^{1,2}(\R^l,\R^l,\Sigma_s(\varphi_s,y),
m_s(y) dy)$ represents the closure of the set of the gradients of functions from
$\C^\infty_0 (\R^l)$ in $\mathbb L^2(\R^l,\R^l, \Sigma_s(\varphi_s,y),m_s(y) dy)$, in which $\mathbb L^2(\R^l,\R^l, \Sigma_s(\varphi_s,y),m_s(y) dy)$ is the Hilbert space of $\R^d$-valued functions (of $y$) in $\R^l$ endowed with the norm $\|f\|_{\Sigma,m}^2\!=\!\int_{\R^l}\!\|f(y)\|^2_{\Sigma_s(\varphi_s,y)}m_s(y)dy$, and for each $(t,x)\in\R_+\times\R^d$, each function $m(\cdot)$ being a probability density function in $\R^l$, $\J_{t,m(\cdot),u}$ is a function of $y$ defined by
$$
\J_{t,m(\cdot),u}(y)=\Pi_{ \Sigma_t(x,\cdot),m(\cdot)}(\Sigma_t(x,y)^{-1}(b_t(x,y) - \di_x \Sigma_t(x,y)/2)),
$$ where $\Pi_{ \Sigma_t(x,\cdot),m(\cdot)}$ maps a function $\phi(y)\in \mathbb L^2_{\rm loc}(\R^l,\R^l, \Sigma_t(x,y),m(y) dy)$, which is the space consisting of functions whose products with arbitrary $\C_0^\infty$-functions belong to $\mathbb L^2(\R^l,\R^l, \Sigma_t(x,y),m(y) dy)$,
into
$$\Pi_{ \Sigma_t(x,\cdot),m(\cdot)}\phi(y)\in \mathbb L_0^{1,2}(\R^l,\R^l, \Sigma_t(x,y),m(y) dy)$$
and satisfies that, for all
$h\in\C^\infty_0 (\R^l)$,
$$
\int_{\R^l} [\nabla h(y)]^\top \Sigma_t(x,y)\Pi_{ \Sigma_t(x,\cdot),m(\cdot)}\phi(y)m(y)dy = \int_{\R^l} [\nabla h(y)]^\top \Sigma_t(x,y)\phi(y)m(y) dy.
$$
If $\phi(y)\in \mathbb L^2(\R^l,\R^l,\Sigma_t(x,y),m(y) dy)$, then $\Pi_{ \Sigma_t(x,\cdot),m(\cdot)}\phi(y)$ is nothing than the orthogonal projection of $\phi$ onto $\mathbb L_0^{1,2}(\R^l,\R^l, \Sigma_t(x,y),m(y) dy)$.
Moreover, it is readily seen from \eqref{eq-I*} that
\begin{equation}\label{eq-28-2}
\text{if }\I^*(\varphi,\mu)<\infty\text{ then }\dot\varphi_s=\int_{\R^l}\frac{F_s(\varphi_s,y)}{\lambda_s(\varphi_s,y)}m_s(y)dy.
\end{equation}
In addition,
the supremum in the last term
in \eqref{eq-I*} is attained at
\begin{equation}\label{eq-hatg}
\hat g(y)=\frac{\nabla_ym_s(y)}{2m_s(y)}-\J_{s,m_s(\cdot),\varphi_s}(y)
\end{equation}
so that
\begin{equation}\label{I*-alternative}
\I^*(\varphi,\mu)=\frac12\int_0^\infty \int_{\R^l}\Big\|\frac{\nabla_ym_s(y)}{2m_s(y)}-\J_{s,m_s(\cdot),\varphi_s}(y)\Big\|^2_{\Sigma_s(\varphi_s,y)}m_s(y)dyds.
\end{equation}
This combining with \eqref{rate-alternative} also show the equality of $\I$ and $\hat I^*$.

Now, we proceed to
 the proof, which contains two main steps.

\para{Step (i): Identify the LD limits at sufficiently regular (dense) points.}

\begin{thm}\label{thm-hat-m}
	Assume that the assumptions of Theorem {\rm \ref{thm-main}} hold.
	Let $\hat \I$ be a LD limit point of
	$\{(X^\eps$,$\mu^\eps)\}_{\eps>0}$
such that $\hat\I(\varphi,\mu) = \infty$ unless $\varphi_0 = \hat x$.
Let $(\hat\varphi,\hat \mu)\in \G$ be such that $\hat\varphi_0=\hat x$ and $\hat\mu(ds,dy)=\hat m_s(dy)ds$, where $\hat m_s(y)$ has the form
\begin{equation}\label{eq-hat-m}\barray
\hat m_s(y)=M_s\bigg(\wdt m_s(y)\hat\eta^2\big(\frac{|y|}{r}\big)+e^{-a |y|}\Big(1-\hat\eta^2\big(\frac{|y|}{r}\big)\Big)\bigg),\earray
\end{equation}
where $\wdt m_s(y)$ is a probability density in $y$
locally bounded away from zero
and belonging to $\C^1(\R^l)$ as a function of $y$ with $|\nabla m_s(y)|$ being locally bounded
in $(s,y)$, and $\hat\eta(y)$ is a nonincreasing $[0,1]$-valued $\C^1_0(\R_+)$-function, with
$y\in\R_+$,
that equals $1$ for $y\in[0,1]$ and equals $0$ for $y \geq 2$; $r > 0$ and $a> 0$, and $M_s$ is the
normalizing constant. For given $\wdt m_ s(y)$, $\hat\eta (y)$, and $r$, there exists $a_0 > 0 $ such
that for all $a > a_0$, $\hat \I(\varphi,\mu)= \I^*(\hat \varphi,\hat \mu)$.
\end{thm}

\para{Technical lemmas.}
Before proving Theorem \ref{thm-hat-m}, we first need some technical lemmas.
For $\beta,h$ as in Section \ref{sec:iden}, we denote
\begin{equation}\label{eq-tau-1}\barray
\ad\tau^{\beta,h}_N(\varphi,\mu)
\\
\ad=\inf\bigg\{t\in\R_+: \int_0^t\int_{\R^l}\|\nabla_yh_s(\varphi_s,y)\|_{\Sigma_s(\varphi_s,y)}\mu(ds,dy)+\sup_{s\in[0,t]}|\varphi_s|+t\geq N\bigg\}.\earray
\end{equation}
Performing
integrating by parts in \eqref{eq-phi-1} yields that
\begin{equation}\label{eq-phi-5}
	\begin{aligned}
		\Phi_t^{\beta,h}(\varphi,\mu):=\int_0^t& \bigg(\beta(s) \dot\varphi_s-[\beta(s)]^\top\int_{\R^l} \frac{ F_s(\varphi_s,y)m_s(y)}{\lambda_s(\varphi_s,y)}\\
		&+\int_{\R^l} [\nabla_y h(s,\varphi_s,y)]^\top \bigg(\frac12\di(\Sigma_s(\varphi_s,y)m_s(y))-b_s(\varphi_s,y)m_s(y)\bigg)dy\\
		&-\frac 12 \int_{\R^l} \|\nabla_yh(s,\varphi_s,y)\|^2_{\Sigma_s(\varphi_s,y)}m_s(y)dy\bigg)ds.
	\end{aligned}
\end{equation}
Let
\begin{equation}\label{eq-theta-1}
\theta_N^{\beta,h}(\varphi,\mu):=\Phi^{\beta,h}_{\tau_N^{\beta,h}}(\varphi,\mu),
\end{equation}
and for each $\delta>0$,
$K_\delta := \{(\varphi,\mu): \hat \I(\varphi,\mu)\leq \delta\}$.
The following are some technical lemmas needed for the proof of Theorem \ref{thm-hat-m}.

\begin{lm}\label{tlm-1}{\rm(\cite[Lemma 7.1]{Puh16})}
{\tt
 (Approximation of $\tau,\theta$)}
  Under the following conditions for the boundedness, and the convergence (uniformly in $K_\delta$) of $\{\beta^i_s\}_{i=1}^{\infty}$, $\{h_s^i(x,y)\}_{i=1}^\infty$ to $\beta_s$, $h_s(x,y)$:
\begin{equation}\label{eq-1lemma-1}\barray
\int_0^N |\beta_s|^2ds+\int_0^N \ess\sup_{(\varphi,\mu\in K_\delta)}\int_{\R^l} |\nabla h_s(\varphi_s,y)|m_s(y)dyds<\infty,\earray
\end{equation}
\begin{equation}\label{eq-1lemma-2}\barray
\lim_{i\to\infty}\ad\int_0^N |\beta_s-\beta_s^i|^2ds\\
\ad+
\lim_{i\to\infty}\sup_{(\varphi,\mu)\in K_{\delta}}\int_0^N\int_{\R^l}|\nabla h_s(\varphi_s,y)-\nabla h_s^i(\varphi_s,y)|^2m_s(y)dyds=0,\earray
\end{equation}
we have the convergence  $$\tau^{\beta^i,h^i}_N(\varphi,\mu)\overset{i\to\infty}{\longrightarrow}\tau^{\beta,h}_N(\varphi,\mu)\text{ and }\theta^{\beta^i,h^i}_N(\varphi,\mu)\overset{i\to\infty}{\longrightarrow}\theta^{\beta,h}_N(\varphi,\mu) \text{ uniformly in }K_\delta.$$
\end{lm}

\begin{lm}\label{tlm-2}{\rm{(\cite[Lemma 7.2]{Puh16})}}
{\tt
(Localizing supremum)} If $h_s(x,y)$
is measurable and belongs to class $\mathbb W^{1,1}_{{\rm loc}}$ in $y$, vanishes when $y$ is outside of some open ball in $\R^l$ locally uniformly in $(s, x)$, and such that
the derivative $Dh_s(x, y)$ is continuous in $(x, y)$ for almost all $s\in\R_+$, and that
$\int_0^N \sup_{x\in\R^d:|x|\leq L }\int_{\R^l} |Dh_s(x, y)|^q dyds <\infty$ for all $q > 1$ and $L > 0$. Then,
$
\sup_{(\varphi,\mu)\in K_{\delta}}(\theta^{\beta,h}_N(\varphi,\mu)-\hat\I(\varphi,\mu))=0,
$
and the supremum is attained.
\end{lm}

\begin{lm}\label{tlm-3}{\rm(\cite[Lemma 7.3]{Puh16})}
{\tt
(Regularities and growth-rate properties of a certain (dense) class)} Assume $m_s(y)$ is an $\R_+$-valued measurable function and is a probability density in $y$ for almost every $s$ and is bounded away from zero on bounded sets of $(s,y)$ and is in $\C^1(\R^l)$, with $|\nabla m_s(y)|$ being locally bounded in $(s,y)$, and $m_s(y) = M_se^{-a|y|} $ ($a > 0$) for all $|y|$ large enough locally uniformly in $s$.
There exists an $a_0$ such that if $a>a_0$,
  there is a $w_s(x,\cdot)$ such that
$\J_{s,m_s(\cdot),u}(\cdot)=\nabla w_s(x,\cdot)$ and satisfies certain regularity and growth-rate properties {\rm\cite[(7.13)-(7.15)]{Puh16}}.
\end{lm}

\begin{proof}[Proof of Theorem {\rm\ref{thm-hat-m}}]
Let $a_0$ and then $\hat w_s(x,y)$ be as in
Lemma \ref{tlm-3}
for $\hat m_s(y)$. Let $\hat\beta=0$ and
$\hat h_s(x,y)=\frac 12 \ln \hat m_s(y)-\hat w_s(x,y).$
Then
$
\nabla \hat h_s(x,y)=\frac{\nabla \hat m_s(y)}{2\hat m_s(y)}-\nabla \hat w_s(x,y).
$
We want to apply
Lemma \ref{tlm-2}
for $\hat\beta$, $\hat h_s(x,y)$. However, $\hat h_s(x,y)$ might not have
a compact support in
$y$. Hence, in order to use
Lemma \ref{tlm-2},
we need to restrict it to a compact set.
Therefore, we shall truncate $\hat h_s(x,y)$. Let
$\eta(t)$ represent an $\R_+$-valued nonincreasing $\C^\infty_0(\R_+)$-function such that $\eta(t) = 1$
for $0 \leq t \leq 1$ and $\eta(t) = 0$ for $t\geq 2$. Let $\hat w^i_s(x,y) =\hat w_s(x,y)\eta(\frac{|y|}i)$ and
$\hat h_s^i(x,y)=\frac 12\eta\Big(\frac{|y|}i\Big) \ln \hat m_s(y)-\hat w_s^i(x,y).$
As in \cite[Lemma 7.4]{Puh16},
we can prove that $\hat h_s(x,y)$ satisfies
the
conditions in
Lemma \ref{tlm-2}.

Next, given $N\in\N$, let $\tau_N^{0,\hat h^i}$ and $\theta_N^{0,\hat h^i}$ be
defined by the respective equations \eqref{eq-tau-1} and \eqref{eq-theta-1} with $\beta=0$ and $h=\hat h^i_s(x,y)$. Since the functions $\beta=0$ and $\hat h^i_s(x, y)$ satisfy the hypothesis in
Lemma \ref{tlm-2}, there exists $(\varphi^{N,i},\mu^{N,i})\in\G$ such
that $\theta_N^{0,\hat h^i}(\varphi^{N,i},\mu^{N,i}) = \hat\I(\varphi^{N,i},\mu^{N,i})$ and $(\varphi^{N,i},\mu^{N,i})\in K_{2N+2}$ for all $i$.
In particular,
$\mu^{N,i}(ds, dy) = m^{N,i}_s (y) dy ds$, where $m^{N,i}_s (\cdot)$ belongs to $\mathcal P(\R^l)$, and the
set $\{(\varphi^{N,i},\mu^{N,i}),i = 1,2,\dots\}$ is relatively compact. Since $\hat \I(\varphi^{N,i},\mu^{N,i}) \geq \I^*(\varphi^{N,i},\mu^{N,i})$ and $\theta_N^{0,\hat h^i}(\varphi^{N,i},\mu^{N,i}) \leq \I^*(\varphi^{N,i},\mu^{N,i})$, one has
\begin{equation}\label{eq-conv-theta}
\theta_N^{0,\hat h^i}(\varphi^{N,i},\mu^{N,i})= \I^*(\varphi^{N,i},\mu^{N,i})= \hat\I(\varphi^{N,i},\mu^{N,i}).
\end{equation}
Extract a convergent subsequence (%which we
still denoting the index by $i$) $\mu^{N,i}\to \mu^N$ in $\C_{\uparrow}(\R_+,\M(\R^l))$ and $\varphi^{N,i}\to\varphi^N$ in $\C(\R_+,\R^d)$.

Because of \eqref{eq-conv-theta} and \eqref{eq-phi-5}, $\I^*(\varphi^{N,i},\mu^{N,i})$ obtains supremum at $\hat h^i_s(x,y)$ when $s\leq \tau^{0,\hat h}_N(\varphi^{N,i},\mu^{N,i})$.
Therefore, by using \eqref{eq-hatg}, we can characterize $m_s^{N,i}$ (noted that $\mu^{N,i}(ds,dy)=m_s^{N,i}(y)dyds$) and then can show that the convergence of  \eqref{eq-1lemma-1} and \eqref{eq-1lemma-2} in the hypothesis of
Lemma \ref{tlm-1}
are satisfied (see \cite[(7.46)-(7.48)]{Puh16}). Thus, by
Lemma \ref{tlm-1},
we have that
$\tau_N^{0,\hat h^i}(\varphi^{N,i},\mu^{N,i}) \to\tau_N^{0,\hat h}(\varphi^N,\mu^N)$ as $i\to\infty$, and that $m^{N,i}_s (y) \to\hat m_s(y)$ in $\mathbb L^1([0,\tau_N^{0,\hat h}(\varphi^N,\mu^N)] \times \R^l)$. Therefore,
$\mu^N(ds,dy) = \hat m_s(y) dyds$ for almost all $s \leq \tau_N^{0,\hat h}(\varphi^N,\mu^N)$.

Using
$\dot \varphi^{N,i}_s= \int_{\R^l} \frac{F_s(\varphi_s^{N,i},y)m^{N,i}_s (y)}{\lambda_s(\varphi_s^{N,i},y)} dy$ due to \eqref{eq-28-2} and
applying \cite[Lemma 6.7]{Puh16}, we obtain
from the convergence of $\varphi^{N,i} \to \varphi^N$ in $\C(\R_+,\R^d)$ and $m^{N,i}_s (y) \to \hat m_s(y)$ in $\mathbb L^1([0,\tau_N^{0,\hat h}(\varphi^N,\mu^N)] \times
\R^l)$ as $i\to\infty$
that $\dot\varphi_s^N = \int_{\R^l} \frac{F_s(\varphi_s^N,y)\hat m_s(y)}{\lambda_s(\varphi^N_s,y)} dy$ a.e. for $s \leq \tau_N^{0,\hat h}(\varphi^N,\mu^N)$. By the uniqueness, $\varphi_s^N = \hat \varphi_s$ for $s \leq \tau_N^{0,\hat h}(\varphi^N,\mu^N)$. As a byproduct, $\dot\varphi_s^{N,i} \to \dot{\hat\varphi}_s$ as $i \to\infty$ a.e. on
$[0,\tau_N^{0,\hat h}(\varphi^N,\mu^N)]$.

We have proved that $\tau_N^{0,\hat h}(\varphi^N,\mu^N) = \tau_N^{0,\hat h}(\hat \varphi,\hat\mu)$ and $\varphi_s^N = \hat \varphi_s$, $\mu^N_s=\hat \mu_s$ for $s \leq \tau_N^{0,\hat h}(\hat\varphi,\hat\mu)$ so that
$\theta_N^{0,\hat h}(\varphi^N,\mu^N) =\theta_N^{0,\hat h}(\hat\varphi,\hat\mu)$. We can show that
$
\theta_N^{0,\hat h}(\varphi^N,\mu^N)=\lim_{i\to\infty} \theta_N^{0,\hat h^i}(\varphi^{N,i},\mu^{N,i}).
$
Therefore, taking the limit in \eqref{eq-conv-theta}, we have
$\I^*(\varphi^N,\mu^N)\geq \theta_N^{0,\hat h}(\varphi^N,\mu^N)\geq \hat \I(\varphi^N,\mu^N),
$ which together with the fact $\I^*\leq\hat \I$ obtained in previous section implies that
$
\I^*(\varphi^N,\mu^N)= \theta_N^{0,\hat h}(\varphi^N,\mu^N)= \hat \I(\varphi^N,\mu^N).
$
Therefore, we have for all $N>0$,
\begin{equation}\label{eq-I-1}
\I^*(\hat\varphi,\hat\mu)\geq \theta_N^{0,\hat h}(\hat\varphi,\hat\mu)=\theta_N^{0,\hat h}(\varphi^N,\mu^N)=\hat \I(\varphi^N,\mu^N).
\end{equation}
From \eqref{eq-I-1}, the fact $\varphi_s^N = \hat \varphi_s$, $\mu^N=\hat \mu$ until $\tau_N^{0,\hat h}(\hat\varphi,\hat\mu)$ and the fact $\hat\I$ is lower semi-continuous and inf-compact, we
obtain
$
\I^*(\hat\varphi,\hat\mu)\geq \hat \I(\hat \varphi,\hat\mu).
$
As a result, we can conclude that $\I^*(\hat\varphi,\hat\mu)=\hat \I(\hat \varphi,\hat\mu)$ for all $(\hat\varphi,\hat\mu)$ satisfying the requirements in Theorem \ref{thm-hat-m}.
\end{proof}

\para{Step (ii): Approximating the LD limits in arbitrary points by regular points.}
Let $\hat \I$ be a LD limit point of
$\{(X^\eps,\mu^\eps)\}_{\eps>0}$ and be such that $\hat\I(\varphi,\mu) = \infty$ unless $\varphi_0 = \hat x$.
In this step, it is proven (see \cite[Theorem 8.1]{Puh16}) that if $\I^*(\varphi,\mu) <\infty$, then there exists a sequence $(\varphi^{(k)},\mu^{(k)})$, whose elements have the properties as in
Theorem \ref{thm-hat-m}
such that $(\varphi^{(k)},\mu^{(k)})\to
(\varphi,\mu)$ and $\I^*(\varphi^{(k)},\mu^{(k)})\to \I^*(\varphi,\mu)$ as $k\to\infty$. Therefore, one has
$ \hat\I(\varphi,\mu)\leq \I^*(\varphi,\mu)=\lim_{k\to\infty}\I^*(\varphi^{(k)},\mu^{(k)})=\lim_{k\to\infty}\hat \I(\varphi^{(k)},\mu^{(k)})\geq\hat\I(\varphi,\mu).
$
Hence, we have obtained desired properties in this Section.

\subsubsection{Completion of
the Proof of Theorem {\rm\ref{thm-main}}} \label{sec:com}

	We will complete the proof of Theorem \ref{thm-main} by removing the restriction that $\hat \I(\varphi,
	\mu) =\infty$ unless $\varphi_0 =\hat x$ in Section \ref{sec:iden}, where $\hat x$ is a preselected element such that $\I_0(\hat x)=0$.
%	Such process
This
can be done similarly as in \cite[Section 9]{Puh16} which will be omitted here.

\section{Fast-Slow Second-Order Systems with General Fast Random Processes}\label{sec:main2}
In this section, we treat \eqref{eq:setup}, in which the fast-varying random process $\xi^\eps_t$
is under a more general setup without specified structure.
 We  need the following assumptions. By a glance, the conditions may seem to be abstract. Nevertheless,
Remark \ref{rem-1-23} illustrates that these assumptions are mild, verifiable, and natural.

\begin{asp}\label{asp-xi1}{\rm
		The functions $F^\eps_t(x,y)$ and $\lambda^\eps_t(x,y)$ are Lipschitz continuous in $x$ locally uniformly in $t$ and globally uniformly in $y$, and $\lambda^\eps_t(x,y)$ is bounded below (uniformly) by a positive constant $\kappa_0$.
		Either
		$F^\eps_t(x,y)$ and $\lambda^\eps_t(x,y)$ have linear growth in $(t,x)$ globally in $y$, i.e., there is a universal constant $\wdt C$ such that
		\begin{equation}\label{eq-asp-xi}
		|F^\eps_t(x,y)|+|\lambda^\eps_t(x,y)|\leq \wdt C(1+|t|+|x|),
		\end{equation}
		or
		$F^\eps_t(x,y)$ and $\lambda^\eps_t(x,y)$ have linear growth in $x$ locally in $y$, i.e., the constant $\wdt C$ in \eqref{eq-asp-xi} is uniformly in bounded sets of $y$
		and $\xi^{\eps}_{t}$ is such that for any $T>0$
		\begin{equation}\label{eq-asp-xi1}
		\lim_{L\to\infty}\limsup_{\eps\to0}\eps\log\PP\Big(\sup_{0\leq t\leq T}|\xi^\eps_{t}|>L\Big)=-\infty.
		\end{equation}
}\end{asp}

\begin{deff}\label{def-localLDP}
	{\rm
		A family of stochastic processes $\{X^\eps\}_{\eps>0}$ is said to  satisfy the local LDP in $\C([0,1],\R^d)$ with rate function $\JJ$, if for any $\varphi\in\C([0,1],\R^d)$,
		\bea
		\ad\lim_{\delta\to 0}\limsup_{\eps\to 0}\eps\log\PP\left(X^\eps\in B(\varphi,\delta)\right)\\
\ad		=\lim_{\delta\to 0}\liminf_{\eps\to 0}\eps\log\PP\left(X^\eps\in B(\varphi,\delta)\right) =-\JJ(\varphi),
		\eea
		where $B(\varphi,\delta)$ is the ball  centered at $\varphi$ with radius $\delta$ in $\C([0,1],\R^d)$.
		$\JJ$ is called local rate function.}
\end{deff}

\begin{asp}\label{asp-xi2}{\rm
		The family of processes $\{Z^\eps\}_{\eps>0}$ given by
		\begin{equation}\label{eq-Z} \barray
		\dot Z^\eps_t=\frac{F^\eps_t(Z^\eps_t,\xi^\eps_{t})}{\lambda^\eps_t(Z^\eps_t,\xi^\eps_{t})},\quad Z^\eps_0=x_0^\eps,\earray
		\end{equation}
		satisfies the local LDP with a rate function $\JJ$.
}\end{asp}

\begin{rem}\label{rem-1-23}
	{\rm Seemingly abstract,
		Assumption \ref{asp-xi2} is not restrictive. In fact, it is
		the LDP for the first-order systems, which is relatively well-understood now.
		For example, the condition is verified when $\xi_t^\eps$ is a (fast-varying) diffusion process
		$$d\xi^\eps_{t}=\frac{1}{\eps}b^\eps_t(X^\eps_t,\xi^\eps_{t})dt+\frac 1{\sqrt \eps}\sigma^\eps_t(X^\eps_t,\xi^\eps_{t})dW_t,$$
		where
			$W_t$ is a standard Brownian motion;
		see \cite{Lip96,Ver99,Ver00}.		 
It is also verified when $\xi_t^\eps$ is a (fast-varying) Markovian switching process
with generator $Q(t)/\eps$ and  $Q(t)$ being a time-inhomogeneous  and irreducible generator of a Markov chain, or $\xi^\eps_t$
is a (fast-varying) jump process having jumps at rate $O(\eps^{-1})$; see \cite{BDG18}.
		 Furthermore, the condition is verified  when $\xi_t^\eps$ has no specific representation but satisfies exponential ergodicity \cite{Gui03}.
		Note also that  condition \eqref{eq-asp-xi1} in Assumption \ref{asp-xi1} is essentially an
		exponential tightness of the fast processes, which is readily verified for
		$\xi^\eps_t$ being diffusion processes or Markovian switching processes.
		When we deal with general fast processes without any specific formulation,
		the assumption on tightness \eqref{eq-asp-xi1} and  Assumption \ref{asp-xi2}
		are very natural.
		Without the tightness and ergodicity of the fast processes, it is unlikely one can
		obtain the averaging and large deviations principles for a fast-slow system under the setting of  general fast processes.
	}
\end{rem}

\begin{rem}{\rm
		We did not assume any regularity and growth-rate conditions
		of the coefficients of the slow component when dealing with \eqref{eq:F-setup}. However, for general fast random process, it seems to be impossible to use the same assumptions
		%due to we did not
		because we do not require
		any structure for the fast process. As a result, the assumptions in this section are stronger  than that of Section \ref{sec:for}. In particular, we need the Lipschitz continuity and growth-rate conditions of $F^\eps_t(x,y)$.
	It is worth noting that we used two totally different approaches for the cases of fast diffusions and general fast-varying processes. If the fast process is a diffusion, thanks to the nice structure of
		martingales, we can identify the rate function after estimating the exponential moment. Therefore, in the first case, after obtaining the
		exponential
		tightness and then relatively LD compactness (see Definition \ref{def:LD}),
		our remaining work is to identify the rate functions. In the general case, we use a different approach that relies on the property that exponential tightness and the local LDP imply the full LDP. In this situation, we need to connect directly the solutions of the second-order and the first-order equations.
}\end{rem}

We are now in a position to present the main theorem. The result is stated next and proof is given in the next section.

\begin{thm}\label{thm-main2}
	Assume that Assumptions {\rm \ref{asp-xi1}} and {\rm\ref{asp-xi2}} hold, that the family $\{x_0^\eps\}_{\eps>0}$ is exponentially tight, and that $\limsup_{\eps\to0}\eps |x_1^\eps|<\infty\a.s$ 	Then, the family $\{X^\eps\}_{\eps>0}$ of \eqref{eq:setup} obeys the LDP in $\C(\R_+,\R^d)$ with rate function $\JJ$.
\end{thm}

\subsection{Proof of Theorem \ref{thm-main2}}\label{sec:prof2}
The proof of this theorem is based on the fact that the exponential tightness and local LDP implies the full LDP.
The following
%definition and
result is well-known in large deviations theory; see, e.g.,
\cite{DS89,DZ98,LP92}.

\begin{prop}\label{prop-local-LDP}
	The exponential tightness and the local LDP for a family $\{X^\eps\}_{\eps>0}$ in $\C([0,1],\R^d)$ with local rate function $\JJ$ imply the full LDP in $\C([0, 1],\R^d)$ for this family with rate function $\JJ$.
\end{prop}

In what follows,
we  prove the LDP of $\{X^\eps\}_{\eps>0}$ in $\C([0,1],\R^d)$.
It will be seen that it can be extended to the space $\C([0,T],\R^d)$ endowed with the sup-norm topology for any $T>0$. As a consequence, the LDP still holds in $\C([0,\infty),\R^d)$, the space of continuous function on $[0,\infty)$ endowed with the local supremum topology.
(This fact follows from the Dawson-G\"artner theorem; see \cite[Theorem 4.6.1]{DZ98}, which states that it suffices  to check the LDPs in $\C([0, T],\R^d)$ for any $T$ in the uniform metric.)
We will still use
$C$ to represent a
generic positive
constant that is
independent of $\eps$. The value $C$ may change at different appearances;
we will specify which parameters it depends on if it is necessary.

\para{Exponential tightness.}
We
aim to prove \eqref{p-cre-11} and \eqref{p-cre-12}.
We have
\begin{equation}\label{eq-xi-X0}
X^\eps_t=x_0^\eps+\int_0^tx_1^\eps e^{-A^\xi_\eps(s)}ds+\dfrac 1{\eps^2}\int_0^t\int_0^se^{-A^{\xi}_\eps(s,r)}F^\eps_s(X^\eps_r,\xi^\eps_{r})dr,
\end{equation}
where for any $0\leq s\leq t\leq 1, \eps>0$,
$A^\xi_\eps(t,s):=\dfrac 1{\eps^2}\int_s^t\lambda^\eps_r(X_r^\eps,\xi^\eps_{r})dr$, $A^\xi_\eps(t)=A^\xi_\eps(t,0).$
So, we can obtain from some direct calculations and Assumption \ref{asp-xi1} that
\begin{equation}\label{eq-xi-X}
|X^\eps_t|\leq |x_0^\eps|+C\eps^2|x_1^\eps|+C\int_0^t \sup_{r\in[0,s]}|F^\eps_r(X^\eps_r,\xi^{\eps}_{r})|ds,
\end{equation}
and by noting further that $\int_s^t e^{-\frac{\kappa_0r}{\eps^2}}dr\leq C \eps^2(1-e^{-\frac{t-s}{\eps^2}})\leq C\eps\sqrt{|t-s|}$, we get
\begin{equation}\label{eq-xi-XS}
|X^\eps_t-X^\eps_s|\leq C\eps |x_1^\eps|\sqrt{|t-s|}+C|t-s|\sup_{r\in[s,t]}|F^\eps_r(X^\eps_r,\xi^{\eps}_{r})|.
\end{equation}
If \eqref{eq-asp-xi} in Assumption \ref{asp-xi1}  holds, \eqref{p-cre-11} follows immediately from \eqref{eq-xi-X} and Gronwall's inequality on noting that $\limsup_{\eps\to0}\eps|x^\eps_1|<\infty$ a.s., $\{x^\eps_0\}_{\eps>0}$ is exponentially tight; and then \eqref{p-cre-12} follows from \eqref{p-cre-11} and \eqref{eq-xi-XS}.

Otherwise, assume that \eqref{eq-asp-xi1} holds.
Let $\wdt C_N$ be constant in \eqref{eq-asp-xi} uniformly in $|y|<N$. We get from \eqref{eq-xi-X} that
$\sup_{t\in[0,1]}|X^\eps_t|\leq C(\wdt C_N+N)e^{\wdt C_N}$ provided $\sup_{t\in[0,1]}|\xi^\eps_{t}|<N$, $|x^\eps_0|<N$.
Therefore, for any $N>0$, for $L>C(\wdt C_N+N)e^{\wdt C_N}$ one has
$$
\PP^\eps(\sup_{t\in[0,1]}|X_t|>L)\leq \PP^\eps(|x^\eps_0|>N)+\PP^\eps(\sup_{t\in[0,1]}|\xi^\eps_{t}|>N).
$$
Letting $L\to\infty$ and $N\to\infty$ and using the logarithm equivalence principle \cite[Lemma 1.2.15]{DZ98}  and \eqref{eq-asp-xi1} in Assumption \ref{asp-xi1}, we get \eqref{p-cre-11}. Thus, we also obtain \eqref{p-cre-12}.

\para{Local LDP.}
It is noted that we do not
assume any structure of $\xi^\eps_{t}$. As a result, we could not use the integration by parts (Lemma \ref{lm-intbypart}) to connect the first-order and the second-order systems.
 Therefore, we will establish a relationship in a local sense.

  For each continuous function $\varphi$, we introduce the auxiliary processes $X^{\eps,\varphi}_t$, the solution of the following equation
\begin{equation}\label{eq-Xphi}
\eps^2\ddot X^{\eps,\varphi}_t=F^\eps_t(\varphi_t,\xi^\eps_{t})-\lambda^\eps_t(\varphi_t,\xi^\eps_{t})\dot X^{\eps,\varphi}_t,\quad X^\eps_0=x^\eps_0, \;\dot X^\eps_0=x^\eps_1,
\end{equation}
and $Z^{\eps,\varphi}_t$, the solution of
\begin{equation}\label{eq-Zphi}
\dot{Z}^{\eps,\varphi}_t=\frac{F^\eps_t(\varphi_t,\xi^\eps_{t})}{\lambda^\eps_t(\varphi_t,\xi^\eps_{t})},\quad Z^{\eps,\varphi}_0=x^\eps_0.
\end{equation}
We have from \eqref{eq-Xphi} and the variation of parameters formula that
\begin{equation}\label{eq-Xphi-1}
X^{\eps,\varphi}_t=x_0^\eps+\int_0^tx_1^\eps e^{-A^\xi_{\eps,\varphi}(s)}ds+\dfrac 1{\eps^2}\int_0^t\int_0^se^{-A^{\xi}_{\eps,\varphi}(s,r)}F^\eps_s(\varphi_r,\xi^\eps_{r})dr,
\end{equation}
where for any $0\leq s\leq t\leq 1, \eps>0$,
$A^\xi_{\eps,\varphi}(t,s):=\dfrac 1{\eps^2}\int_s^t\lambda^\eps_r(\varphi_r,\xi^\eps_{r})dr$, and $ A^\xi_{\eps,\varphi}(t)=A^\xi_{\eps,\varphi}(t,0).$
From the fact that $A^\xi_\eps(s,r),A^\xi_{\eps,\varphi}(s,r)\geq \frac{\kappa_0 (s-r)}{\eps^2}$,
and the property of $\lambda$,
we  obtain that
\begin{equation}\label{3.5}
\begin{aligned}
\abs{e^{-A^\xi_\eps(s,r)}-e^{-A^\xi_{\eps,\varphi}(s,r)}}
\leq&
C e^{\frac{-\kappa_0 (s-r)}{\eps^2}}\cdot\frac{1}{\eps^2}\int_r^s\abs{X^\eps_r-\varphi_r}dr\\
\leq& Ce^{\frac{-\kappa_0 (s-r)}{\eps^2}}\cdot\frac{s-r}{\eps^2}\cdot\sup_{r\in[0,s]}|X^\eps_r-\varphi_r|.
\end{aligned}
\end{equation}
A change of variable  leads to
\begin{equation}\label{eq-B2-2}
\int_0^s \exp\left\{\frac{-\kappa_0 (s-r)}{\eps^2}\right\}\cdot\frac{s-r}{\eps^2}ds=\eps^2\int_0^{\frac s{\eps^2}}e^{-\kappa_0 r}rdr
\leq C\eps^2.
\end{equation}
Therefore,
we obtain from the Lipschitz property of the coefficients and \eqref{3.5} that
\begin{equation}\label{3.6}
\begin{aligned}
\ad
\int_0^s \abs{e^{-A^\xi_\eps(s,r)}F^\eps_r(X^\eps_r,\xi^\eps_{r})-e^{-A^\xi_{\eps,\varphi}(s,r)}F^\eps_r(\varphi_r,\xi^\eps_{r})}dr
\\
\aad \ \leq \int_0^s e^{-A^\xi_\eps(s,r)}\abs{F^\eps_r(X^\eps_r,\xi^\eps_{r})-F^\eps_r(\varphi_r,\xi^\eps_{r})}dr\\
%\\&\;\;\;\;\;
\aad \quad +\int_0^s \abs{F^\eps_r(\varphi_r,\xi^\eps_{r})}\abs{e^{-A^\xi_\eps(s,r)}-
	e^{-A^\xi_{\eps,\varphi}(s,r)}}dr\\
\aad \ \leq C\eps^2\sup_{0\leq r\leq s}\abs{X^\eps_r-\varphi_r}+C\eps^2\sup_{r\in[0,s]}|F^\eps_r(\varphi_r,\xi^\eps_{r})|\sup_{0\leq r\leq s}\abs{X^\eps_r-\varphi_r}.
\end{aligned}
\end{equation}
Combining \eqref{eq-xi-X0}, \eqref{eq-Xphi-1}, and applying \eqref{3.6} and noting that $\limsup_{\eps\to0}\eps|x^\eps_1|<\infty$ leads to
\begin{equation}\label{eq-XXphi}
\sup_{s\in[0,t]}|X^\eps_s-X^{\eps,\varphi}_s|\leq C\eps\sup_{s\in[0,t]}|X^\eps_s-\varphi_s|+C\sup_{r\in[0,t]}|F^\eps_r(\varphi_r,\xi^\eps_{r})|\int_0^t\sup_{r\in[0,s]}|X^\eps_r-\varphi_r|ds.
\end{equation}
From \eqref{eq-Xphi} and \eqref{eq-Zphi}, we obtain that
\begin{equation}\label{eq-XZphi}
\barray
|X^{\eps,\varphi}_t-Z^{\eps,\varphi}_t|\ad=\Big|\int_0^t\frac{\eps^2\ddot{X}^{\eps,\varphi}_r}{\lambda^\eps_r(\varphi_r,\xi^\eps_{r})}dr\Big|
\leq C\eps^2\sup_{s\in[0,t]}|\dot X^{\eps,\varphi}_t|\\
\ad\leq C\eps^2(|x^\eps_1|+\sup_{r\in[0,t]}|F^\eps_r(\varphi_r,\xi^\eps_{r})|).
\earray
\end{equation}
One also gets from \eqref{eq-Z}, \eqref{eq-Zphi}, and the Lipschitz continuity of $F^\eps,\lambda^\eps$ that
\begin{equation}\label{eq-ZZPhi}
\sup_{s\in[0,t]}|Z^{\eps,\varphi}_s-Z^\eps_s|\leq C\Big(\sup_{r\in[0,t]}|F^\eps_r(\varphi_r,\xi^\eps_{r})|+\sup_{r\in[0,t]}|
\lambda^\eps_r(\varphi_r,\xi^\eps_{r})|\Big)\int_0^t\sup_{r\in[0,s]}|Z^\eps_r-\varphi_r|ds.
\end{equation}
Now,
if \eqref{eq-asp-xi} in Assumption \ref{asp-xi1}  holds, combining \eqref{eq-XXphi}, \eqref{eq-XZphi}, and \eqref{eq-ZZPhi}, we get
\bea
\sup_{s\in[0,t]}|X^\eps_s-\varphi_s|\ad\leq C\eps\sup_{s\in[0,t]}|X^\eps_s-\varphi_s|+C\eps(1+\sup_{r\in[0,1]}|\varphi_r|)\\
\aad+C(1+\sup_{r\in[0,1]}|\varphi_r|)\sup_{r\in[0,1]}|Z^\eps_r-\varphi_r|\\ \aad +C(1+\sup_{r\in[0,1]}|\varphi_r|)\int_0^t\sup_{r\in[0,s]}|X^\eps_r-\varphi_r|ds,
\eea
and
\bea
\sup_{s\in[0,t]}|Z^\eps_s-\varphi_s|\ad \leq C\eps(1+\sup_{r\in[0,1]}|\varphi_r|)+C(1+\sup_{r\in[0,1]}|\varphi_r|)\sup_{r\in[0,1]}|Z^\eps_r-\varphi_r|\\
\ad \ +C(1+\sup_{r\in[0,1]}|\varphi_r|)\int_0^t\sup_{r\in[0,s]}|X^\eps_r-\varphi_r|ds.
\eea
Thus, for small $\eps$, one has
\begin{equation}\label{eq-supXphi}
\barray \ad
\sup_{t\in[0,1]}|X^\eps_t-\varphi_t|\leq C_{1,\varphi}\eps+C_\varphi\sup_{t\in[0,1]}|Z^\eps_t-\varphi_t|,
 \ \hbox{ and } \\
\ad
\sup_{t\in[0,1]}|Z^\eps_t-\varphi_t|\leq C_{2,\varphi}\eps+C_\varphi\sup_{t\in[0,1]}|X^\eps_t-\varphi_t|,
\earray
\end{equation}
for some constants $C_{1,\varphi}$, $C_{2,\varphi}$ depending only on $\sup_{r\in[0,1]}|\varphi_r|$ and independent of $\eps$.
So, for any $\delta>0$ we have from \eqref{eq-supXphi}
that
\bea \ad
\PP^\eps(\sup_{t\in[0,1]}|X^\eps_t-\varphi_t|<\delta)\leq \PP^\eps(\sup_{t\in[0,1]}|Z^\eps_t-\varphi_t|<2\delta C_{2,\varphi}),\;\forall \eps<\delta/C_{2,\varphi}, \ \hbox{ and } \\
\ad
\PP^\eps(\sup_{t\in[0,1]}|X^\eps_t-\varphi_t|<\delta)\geq \PP^\eps\Big(\sup_{t\in[0,1]}|Z^\eps_t-\varphi_t|<\frac{\delta} {2C_{1,\varphi}}\Big),\;\forall \eps<\delta/C_{1,\varphi}.
\eea
Therefore, the local LDP of $\{X^\eps\}_{\eps>0}$ follows directly from the local LDP of $\{Z^\eps\}_{\eps>0}$.

If \eqref{eq-asp-xi} in Assumption \ref{asp-xi1} only holds locally and \eqref{eq-asp-xi1} holds, then in the event $\sup_{t\in[0,1]}|\xi^\eps_{t}|<N$, we still have \eqref{eq-supXphi}
with $C_{1,\varphi}$, $C_{2,\varphi}$ replaced by $C_{1,\varphi,N}$, $C_{2,\varphi,N}$ depending only on $\varphi,N$.
So, for any $\delta>0$ one also has that $\forall \eps<\delta/C_{1,\varphi,N}\wedge \delta/C_{2,\varphi,N}$,
\bea
\ad\PP^\eps(\sup_{t\in[0,1]}|X^\eps_t-\varphi_t|<\delta)\\\ad\leq \PP^\eps(\sup_{t\in[0,1]}|Z^\eps_t-\varphi_t|<2\delta C_{2,\varphi,N})+\PP^\eps(\sup_{t\in[0,1]}|\xi^\eps_{t}|>N),
\eea
and
\bea
\ad
\PP^\eps(\sup_{t\in[0,1]}|X^\eps_t-\varphi_t|<\delta)\\\ad\geq \PP^\eps\Big(\sup_{t\in[0,1]}|Z^\eps_t-\varphi_t|<\frac{\delta} {2C_{1,\varphi,N}}\Big)-\PP^\eps(\sup_{t\in[0,1]}|\xi^\eps_{t}|>N).
\eea
By letting $\eps\to0$, $N\to\infty$, and $\delta\to0$ and using the logarithm equivalence principle \cite[Lemma 1.2.15]{DZ98} and \eqref{eq-asp-xi1} in Assumption \ref{asp-xi1}, we obtain the local LDP for $\{X^\eps\}_{\eps>0}$.
Therefore, the proof of Theorem \ref{thm-main2} is complete.

\section{Examples}\label{sec:example}
In this section, we consider some examples drawn from physics to illustrate our formulation and results in these  cases.

\subsection{Stochastic Acceleration with Small-mass Particles}
Stochastic acceleration considers motions of a net of particles in a net of random force fields, which is described by the Newton's law as
$
\ddot x_\eps(t)=\wdt F_\eps(t,\omega,x_\eps(t),\dot x_\eps(t),\chi_\eps(t)),
$
where $\wdt F_\eps$ denotes the random force fields.
Such
models were
considered by Kesten and Papanicolaou in \cite{KP79,KP80} and references therein.
Note that small and large are relative terms.
Here we focus on
small-mass particles, by
which we mean that
the Reynolds number is small (see e.g., \cite{Pur77} for a definition)
so that inertial effects  are negligible
compared to the damping force, or the ratio `inertial effects/damping force' is parameterized by $\eps\ll 1$.
Therefore, random force field $\wdt F_\eps$ can be written as
$\wdt F_\eps=F_\eps-\frac{\lambda_\eps}\eps$, and the motion is described by
$
\ddot x_\eps(t)=F_\eps(t,x_\eps(t),\chi_\eps(t))-
\frac{\lambda_\eps(t,x_\eps(t),\chi_\eps(t))}
{\eps}\dot x_\eps(t).
$
Now, by scaling $X^\eps_t:=x_\eps(t/\eps)$,
and $\xi_{t}^\eps:=\chi_\eps(t/\eps)$,
the system can be rewritten as \eqref{eq:setup}, i.e.,
\begin{equation}\label{eq:setup-0}
	\eps^2\ddot X^\eps_t=F^\eps_t(X^\eps_t,\xi^\eps_{t})-\lambda^\eps_t(X^\eps_t,\xi^\eps_{t})\dot X^\eps_t,\quad X^\eps_0=x_0\in\R^d,\quad\dot{X}^\eps_0=x_1\in \R^d.
\end{equation}
The above illustrates how the fast-varying process (or fast-varying random environment) $\xi^\eps_t=\chi_\eps(t/\eps)$ comes in. To further demonstrate,
 we consider some common models of the fast-varying processes
 for random environment $\xi_t^\eps$ and illustrate our results.

\subsubsection{Fast-Varying Diffusion} Consider the case fast-varying process $\xi_t^\eps$ is modeled as a (fast) diffusion, which is common in modeling stochastic process in physical phenomena, i.e.,
\begin{equation}\label{eq-xi}
\dot{\xi}^\eps_t=\dfrac 1\eps  b^\eps_t(X^\eps_t,{\xi}^\eps_t)+\dfrac 1{\sqrt\eps}\sigma^\eps_t(X^\eps_t,{\xi}^\eps_t)\dot W_t,\quad \xi^\eps_0=\xi_0\in\R^l.
\end{equation}
In this situation, stochastic acceleration
%system
\eqref{eq:setup-0}-\eqref{eq-xi} become coupled second-order SDEs \eqref{eq:F-setup}.
The LDP for stochastic acceleration in this case is established in Theorem \ref{thm-main} with rate function given in variational form \eqref{eq-rate} or
representation \eqref{rate-alternative}.
It is important to note that we establish LDP for stochastic acceleration
assuming neither Lipschitz continuity nor linear growth-rate of $F^\eps$; see Assumptions \ref{asp-2}, \ref{asp-3}, and Remarks \ref{rem-11}.

\begin{thm}
	Under Assumptions \ref{asp-2}, \ref{asp-3}, and \eqref{eq-asp-conv}, stochastic acceleration under fast-varying diffusion environment \eqref{eq:setup-0}-\eqref{eq-xi} obeys LDP with the rate function given in \eqref{eq-rate}.
\end{thm}

\subsubsection{Fast-Varying Jumps}\label{sec:ex-jump}
Consider the case $\xi_t$ is a jump process taking finite values in $\M=\{1,\dots,|\M|\}$,
where $|\M|$ denotes the cardinality of the set $\M$.
Similar to \cite{BDG18}, the evolution of the jump fast component is constructed through a jump intensity function $c(x,y)=c_y(x) : \R^d \times\M\to [0, \infty)$ and a transition probability function $r(x,y,y')=r_{yy'}(x) : \R^d\times\M \times \M \to [0, 1]$, both of which are coupled with $X^\eps$.
To be self-contained, we describe the construction of jump processes $\xi_t^\eps$ as follows.

Assume that for all $(x,y)\in\R^d\times\M$, $\sum_{y'\in\M} r_{yy'}(x)=1, r_{yy}(x) = 0.$
Let $\zeta=\sup_{(x,y)\in\R^d\times\M}c_y(x)+1$,
$E_{yy'}(x)=[0,c_y(x)r_{yy'}(x)]$ for all $(x,y,y')\in\R^d\times\M\times\M$, $y\neq y'$, and
$\mathbb T =: \{(y,y') \in\M\times\M : r_{yy'}(x) > 0$  for some $x \in\R^d\}$.
For
$(i,j)\in\mathbb T$, let $\bar N_{ij}$ be a Poisson random measure on $[0; \zeta] \times [0,T] \times \R_+$ with intensity measure $\mu_{\zeta}\otimes \mu_T \otimes \mu_{\infty}$, where $\mu_T$ and $\mu_\infty$ denote the Lebesgue measures on $[0, T]$
and $\R_+$, respectively, such that
for $t\in[0,T]$,
$
\bar N_{ij}(A \times [0, t] \times B) - t\mu_\zeta(A)\mu_\infty(B)
$
is a $\F_t$-martingale for all $A\in \mathcal B[0,\zeta]$ and $B\in \mathcal B(\R_+)$ with $\mu_\infty(B) < 1$.
Then, we define
$
N^{\eps^{-1}}
_{ij} (dr\times dt) = \bar N_{ij}(dr \times dt \times [0,\eps^{-1}])
$, which
is a Poisson random measure on $[0, \zeta]\times [0, T]$ with intensity
$\eps^{-1}\mu_\zeta\otimes\mu_T$.
The processes $(N^{\eps^{-1}}_{ij})_{(i,j)\in\mathbb T}$ are taken to be mutually independent. We will assume that for $0 \leq s \leq t \leq T$,
$
\{
N^{\eps^{-1}}_{ij}(A \times (s; t] \times B): A \in \mathcal B[0,\zeta], B \in\mathcal B(\R_+), (i,j)\in
\mathbb T\}
$
is independent of $\F_s$.
Now, we consider the following stochastic acceleration with fast-varying jumps
\begin{equation}\label{qe-exp-2}
	\begin{cases}
		\eps^2\ddot{X}^\eps_t=F^\eps(X^\eps_t,Y^\eps_t)-\lambda^\eps(X^\eps_t,Y^\eps_t)\dot{X}^\eps_t,\\
			dY^\eps_t=\sum_{(i,j)\in\mathbb T}\int_{r\in[0,\zeta]}(j-i)\1_{\{Y^\eps(t-)=i\}}\1_{E_{ij}(X^\eps_t)}(r)N_{ij}^{\eps^{-1}}(dr\times dt),\\
		X^\eps_0=x_0\in\R^d,\quad\dot{X}^\eps_0=x_1\in \R^d,\quad Y_0^\eps=y_0\in\M.
	\end{cases}
\end{equation}
According to \cite{BDG18}, we make following assumption for the jump process.
\begin{asp}\label{asp-exp-2}
	The function $c$ is a bounded and
	there exists a finite constant $C>0$ such that for all $y,y'\in \M$ and $x_1, x_2\in \R^d$,
	$$
	|c_{y}(x_1) - c_y(x_2)|+|r_{yy'}(x_1) - r_{yy'}(x_2)|
	\leq C|x_1 - x_2|.
	$$
	Moreover,
	$$	 \inf_{x\in\R^d}\min_{y,z\in\M}\sum_{n=1}^{|\M|}r^n_{yz}(x)>0,\quad \inf_{x\in\R^d}\min_{y\in\M}c_y(x)>0,\quad
	\inf_{x\in\R^d}\min_{(y,y')\in\mathbb T}r_{yy'}(x)>0.
	$$
\end{asp}

The rate function for the LDP of \eqref{qe-exp-2} is constructed as follows.
For $\psi=(\psi(j))_{j\in\M}$, with $\psi_j:[0, \zeta] \to \R_+$ being a measurable map for every $j$, define
$$
\Phi^{\psi}_{ij}(x) =
\begin{cases}
	\int_{E_{ij}(x)}\psi_j(z)\mu_\zeta d(z),\text{ if }i\neq j,\\
	-\sum_{y:y\neq j}\Phi^{\psi}_{jy}(x),\text{ if } i=j,
\end{cases}
$$
and
$
\mathcal R =\{v=(v_{ij})_{(i,j)\in\mathbb T}, v_{ij}: [0,1]\times[0,\zeta]\to\R_+\text{ is measurable for all } (i,j)\in\mathbb T\}.
$
For $\varphi\in C([0,1],\R^d)$, let $\mathcal V(\varphi)$ be the collection of all
$$
\big(u=(u_i),v=(v_{ih}),\pi=(\pi_i)\big)\in\mathbb M([0,1]:\R^d)^{|\M|}\times\mathcal R\times \mathbb M([0,1]:\mathcal P(\M)),
$$
where $\mathbb M([0, 1] : \mathcal P(\M))$,
$\mathbb M([0,1] : \R^d)$
denote the space of measurable maps from
$[0, 1]$ to $\mathcal P(\M)$ and from $[0,1]$ to $\R^d$, respectively, with $\mathcal P(\M)$ being the space of probability measures on $\M$ equipped with the topology of weak convergence], such that $\int_0^1\|u_i(s)\|^2\pi_i(s)ds<\infty$ for each $i\in\M$, and
$$
\varphi_t=x_0+\sum_{j\in\M}\int_0^t \frac{F(\varphi_s,j)}{\lambda(\varphi_s,j)}\pi_j(s)ds;\;
\sum_{j\in\M}\pi_j(s)\Phi^{v_{j\cdot}(s,\cdot)}_{ji}(\varphi_s)=0, a.e.\; s\in[0,1], \forall i\in\M.
$$
Combining Theorem \ref{thm-main2} and \cite{BDG18} yields the following result.

\begin{thm}\label{thm-ex-jump}
	Assume Assumptions {\rm\ref{asp-xi1}} and  {\rm\ref{asp-exp-2}} hold.
	Then the family of processes $\{X^\eps\}_{\eps>0}$
	in stochastic acceleration system with fast-varying jump \eqref{qe-exp-2}
	satisfies the LDP
	with the rate function $\I$ given by
	\begin{equation}\label{eq-exp2-rate}
		\begin{aligned}
		\I(\varphi)=\displaystyle\inf_{(u,v,\pi)\in\mathcal V(\varphi)}\Bigg\{\sum_{i\in\M}&\frac 12\int_0^1\|u_i(s)\|^2\pi_i(s)ds\\
		&+\sum_{(i,j)\in\mathbb T}\int_{[0,\zeta]\times[0,1]}\ell(v_{ij}(s,z))\pi_i(s)\mu_\zeta(dz)ds\Bigg\},
		\end{aligned}
	\end{equation}
	where $\ell(x) = x \ln x - x + 1$.
\end{thm}

\subsection{Li\'enard equation with relaxation oscillations}
The Li\'enard equations, named after physicist Alfred-Marie Li\'enard, have been extensively studied in the literature of ordinary differential
equations. During the development of radio and vacuum tubes, the Li\'enard equations were
 used to model oscillating circuits.
 These equations were also used in
 mechanical systems in physics and engineering.
In the exploration of radio and vacuum tube technologies,
much attention was devoted
to the study of Li\'enard equations
and such equations with relaxation oscillations.
A notable important equation is the following
\begin{equation}\label{eq-Li}
	\frac 1{\nu^2}\ddot{x}^\nu(t)=g(x^\nu(t))-\kappa \dot{x}^\nu(t),
\end{equation}
where $\nu\gg 1$ is a large number, $\kappa$ is a positive constant, and $g$ is a function.
Equation \eqref{eq-Li} has been studied in detail in \cite{Nar93}
with the motivation from the
 familiar van der Pol equation  \cite{Poe27}.
Its variations can also be found in
 \cite{XY15} and references therein.

Now, we consider the case that the environment is perturbed by random factors so that the function $g$ and coefficient $\kappa$ depend on a random process, which varies very fast.
Such a fast-slow setting is natural as multiscale systems arise in many problems in various fields. For example, many processes  (e.g., signals, cellular processes) are inherently multiscale in nature with reactions occurring at varying speeds.
As a result, we consider the following  Li\'enard equation with relaxation oscillations in a fast-varying random environment
\begin{equation}\label{eq-Li-1}
	\frac 1{\nu^2}\ddot{x}^\nu(t)=g(x^\nu(t),\xi^\nu(t))-\kappa(\xi^\nu(t)) \dot{x}^\nu(t),
\end{equation}
where $\xi^\nu(t)$ is a (fast-varying) random process, which interacts with $x^\nu(t)$.
In particular, the time-scale separation comes from applications; see for example, \cite{Nar93} and references therein.
Using our results, we can establish LDP for the family of solutions $\{x^\nu(\cdot)\}_{\nu\gg 1}$ of \eqref{eq-Li-1}.
(i) If $\xi^\nu(t)$ has the form of a (fast) diffusion, LDP of $\{x^\nu(\cdot)\}_{\nu\gg 1}$ is established by Theorem \ref{thm-main} without any assumption about Lipschitz continuity of $g$. (ii) If $\xi^\nu(t)$ is a (fast) jump process, LDP of $\{x^\nu(\cdot)\}_{\nu\gg 1}$ can be obtained as in Section \ref{sec:ex-jump} (Theorem \ref{thm-ex-jump}).
For brevity, we only state the results without the verbatim derivations.

\begin{thm}
	\begin{itemize}
		\item [\rm{(i)}] Assume $d{\xi}^\nu(t)=\nu  b^\nu_t(x^\nu(t),{\xi}^\nu(t))+\sqrt{\nu}\sigma^\nu_t(x^\nu_t,{\xi}^\nu(t))dW(t), \xi^\nu_0=\xi_0\in\R^l$.
			Under Assumptions \ref{asp-2}, \ref{asp-3}, and \eqref{eq-asp-conv}, $\{x^\nu(\cdot)\}_{\nu\gg 1}$ satisfies LDP with the rate function
		\begin{equation}\label{eq-rate-lie}
			\begin{aligned}
\aad
\I_X(\varphi)= \I_0(\varphi_0)+\int_0^\infty \sup_{\beta\in\R^d} \Big[
\beta^\top\dot\varphi_s -\sup_{m\in\mathcal P(\R^l)}
\Big(\beta^\top\int_{\R^l}\frac{g(\varphi_s,y)}{\kappa(y)}m(y)dy\\
				\aad \
+\sup_{h\in \C_0^1(\R^l)}\int_{\R^l}\!
\Big([\nabla h(y)]^\top \Big(\frac12\di_y\big(\Sigma_s(\varphi_s,y)m(y)\big)-b_s(\varphi_s,y)m(y)\Big)\\
				\aad \
				-\frac 12\|\nabla h(y)\|^2_{\Sigma_s(\varphi_s,y)}m(y)
\Big)dy
\Big)\Big]
ds,
			\end{aligned}
		\end{equation}
	if $\varphi$ is absolutely continuous;
		otherwise, $\I_X(\varphi) = \infty$.
		
		\item [\rm{(ii)}] Assume $d\xi^\nu(t)=\sum_{(i,j)\in\mathbb T}\int_{r\in[0,\zeta]}(j-i)\1_{\{\xi^\nu(t-)=i\}}\1_{E_{ij}(x^\nu_t)}(r)N_{ij}^{\nu}(dr\times dt),$ where $N^\nu_{ij}$ is
a		Poisson random measure
with (fast) intensity rate $O(\nu)$
		constructed precisely as in Section \ref{sec:ex-jump}.
	Under Assumptions {\rm\ref{asp-xi1}} and  {\rm\ref{asp-exp-2}},
	The	$\{x^\nu\}_{\nu\gg 1}$
		satisfies the LDP
		with the rate function $\I$ given by
		\begin{equation}\label{eq-exp2-rate-lie}
			\begin{aligned}
				\I(\varphi)=\displaystyle\inf_{(u,v,\pi)\in\mathcal V(\varphi)}
&\Big\{\sum_{i\in\M}
\frac 12\int_0^1\|u_i(s)\|^2\pi_i(s)ds\\
				&\ +\sum_{(i,j)\in\mathbb T}\int_{[0,\zeta]\times[0,1]}\ell(v_{ij}(s,z))\pi_i(s)\mu_\zeta(dz)ds\Big\},
			\end{aligned}
		\end{equation}
		where $\ell(x) = x \ln x - x + 1$, and $\mathcal V(\varphi)$ be the collection of all $\big(u=(u_i),v=(v_{ih}),\pi=(\pi_i)\big)$
		such that $\int_0^1\|u_i(s)\|^2\pi_i(s)ds<\infty$ for each $i\in\M$, and
		$
		\varphi_t=x_0+\sum_{j\in\M}\int_0^t \frac{g(\varphi_s,j)}{\kappa(j)}\pi_j(s)ds$,
			and $\sum_{j\in\M}\pi_j(s)\Phi^{v_{j\cdot}(s,\cdot)}_{ji}(\varphi_s)=0, a.e.\; s\in[0,1], \forall i\in\M.
		$
	\end{itemize}
\end{thm}


\begin{thebibliography}{99}

\bibitem{BEM22}
H. Bessaih,
Y. Efendiev,
 R. Maris,
 Stochastic homogenization of a convection-diffusion equation {\it SIAM J. Math. Anal.} {\bf 53} (2021), no. 3, 2718--2745.

\bibitem{BDG18} A.
Budhiraja,
P.
Dupuis,
A. Ganguly, Large deviations for small noise diffusions in a fast Markovian environment, {\it Electron. J. Probab.} {\bf 23} (2018), no. 112, 1--33.

\bibitem{CCKW} X. Chen, Z.-Q. Chen, T. Takashi, and J. Wang,
 Homogenization of symmetric stable-like processes in stationary ergodic media, {\it SIAM J. Math. Anal.} {\bf 53} (2021),  2957–3001.



\bibitem{CF05} Z. Chen, M.I. Freidlin, Smoluchowski--Kramers approximation and exit problems, {\it Stoch. Dyn.} {\bf 5} (2005), 569--585.

\bibitem{Cheng} L. Cheng, R. Li, and W. Liu, Moderate deviations for the Langevin equation with strong damping,  {\it J.
Statist. Phys.} {\bf 170} (2018),  845--861.


\bibitem{Freidlin} S. Cerrai and M. Freidlin,  Large deviations for the Langevin equation with strong damping, {\it J. Statist. Phys.} {\bf 161} (2015), 859--875.

\bibitem{CG20}
S. Cerrai
and
N.
Glatt-Holtz, On the convergence of stationary solutions in the Smoluchowski-Kramers approximation of infinite dimensional systems, {\it J. Funct. Anal.} {\bf 278} (2020), no. 8, 108421, 38 pp.

\bibitem{CX22}
S. Cerrai and
G.
Xi, A Smoluchowski-Kramers approximation for an infinite dimensional system with state-dependent damping, {\it Ann. Probab.} {\bf 50} (2022), no. 3, 874--904.



\bibitem{CC18}
 F. Cipriano,
 T. Costa,
 A large deviations principle for stochastic flows of viscous fluids. {\it J. Differential Eqs.} {\bf 264} (2018),
  5070--5108.



\bibitem{DZ98} 
A. Dembo, O. Zeitouni, {\it Large Deviations Techniques and Their Applications}, 2nd ed.
Jones and Bartlett, Boston, 1998.

\bibitem{DS89}
J.D. Deuschel and D.W. Stroock, {\it Large Deviations}, Academic Press, San Diego, 1989.

\bibitem{DE11}
P.
Dupuis,
R.
Ellis,
{\it A Weak Convergence Approach to the Theory of Large Deviations}, John Wiley \& Sons, 2011.

\bibitem{FeKu}
J. Feng, T. Kurtz, {\it Large Deviations for Stochastic Processes},
Amer. Math Soc., Prov.,
2006.


\bibitem{FGS12}
G. Fibich,
A. Gavious,
and
E. Solan, Averaging principle for second-order approximation of heterogeneous models with homogeneous models, {\it PNAS},
 {\bf 109} (48) (2012), 19545--19550. 

\bibitem{Fre04} M.I. Freidlin, Some remarks on the Smoluchowski-Kramers approximation, {\it J. Statist. Phys.} {\bf 117} (2004), 617--634.

\bibitem{Fre84} M.I. Freidlin and A.D. Wentzell, {\it Random Perturbations of Dynamical Systems}, Springer-Verlag, New York, 1984.



\bibitem{Gao} P. Gao,
Averaging principle for complex Ginzburg--Landau equation perturbated by mixing random forces, {\it SIAM J. Math. Anal.} {\bf 53} (2021), 32--61.



\bibitem{Gui03} A. Guillin, Averaging principle of SDE with small diffusion: Moderate deviations, {\it Ann. Probab.}
{\bf 31} (2003), 413--443.


\bibitem{HLL} W. Hong, S. Li, and W, Liu,
Freidlin--Wentzell type large deviation principle for multiscale locally monotone SPDEs,
{\it SIAM J. Math. Anal.} {\bf 53},  (2021),  6517--6561.


\bibitem{Yinhe} Q. He, G. Yin, Large deviations for multi-scale Markovian switching systems with a small diffusion. {\it Asymptot. Anal.} {\bf 87} (2014),
123--145.

\bibitem{IMS}
S.A. Isaacson, J. Ma, and K. Spiliopoulos,
Mean field limits of particle-based stochastic reaction-diffusion models,
{\it SIAM J. Math. Analy.}, {\bf 54}  (2022), 453--511.

\bibitem{KP79} H. Kesten, G.C. Papanicolaou, A limit theorem for turbulent diffusion, {\it Comm. Math. Phys.} {\bf 65} (1979), 97-128.

\bibitem{KP80} H. Kesten, G C. Papanicolaou, A limit theorem for stochastic acceleration, {\it Comm. Math. Phys.} {\bf 78} (1980), 19-63.


\bibitem{KhY}
R. Khasminskii and G. Yin,
On averaging principles: An asymptotic expansion approach,
 {\it SIAM J. Math. Anal.}, {\bf 35} (2004),
1534--1560.


\bibitem{Kif09}
Y. Kifer,
Large deviations and adiabatic transitions for dynamical systems and Markov processes in fully coupled averaging, {\it Mem. Amer. Math. Soc.} {\bf 201}  (2009).


\bibitem{Kus20} H.J. Kushner, P. Dupuis, {\it Large deviations estimates for systems with small (wideband) noise effects and applications to stochastic systems and conmunication theory}, Probability Theory and Applications,
659--662, De Gruyter, 2020.

\bibitem{LP92}
R.S. Liptser, A.A. Puhalskii, Limit theorems on large deviations for semimartingales, {\it Stochastics Stochastics Rep.} {\bf 38} (1992), 201--249.

\bibitem{Lip96} 
R.S. Liptser, Large deviations for two scaled diffusions. {\it Probab. Theory Related Fields},
{\bf 106} (1996), 71--104.

\bibitem{Mao97} X. Mao, {\it Stochastic Differential Equations and Their Applications}, Horwood,
    Chichester, 1997.


\bibitem{Nar93}
K.
Narita,
Asymptotic behavior of solutions of SDE for relaxation oscillations. {\it SIAM J. Math. Anal.} {\bf 24} (1993), 172--199.

\bibitem{NY-JMP} N. Nguyen, G. Yin, A class of Langevin equations with Markov switching involving strong damping and fast switching, {\it J. Math. Phys.} {\bf 61} (2020), 063301,
18 pp.

\bibitem{NY-JMP2} N. Nguyen, G. Yin, Large deviations principles for Langevin equations in random environment and applications, {\sl J. Math. Phys.}
{\bf 62} (2021), 083301, 26 pp.





\bibitem{PIX21}
 B. Pei,
  Y. Inahama,
  Y. Xu,
  Averaging principle for fast-slow system driven by mixed fractional Brownian rough path. {\it J. Differential Eqs.} {\bf 301} (2021), 202--235.



\bibitem{Puh16} A.A. Puhalskii, On large deviations of coupled diffusions with time scale separation, {\it Ann. Probab.} {\bf 44} (2016), 3111--3186.


\bibitem{Pur77} E.M. Purcell, Life at low Reynolds number, {\it Amer. J. Phys.} {\bf 45}, 3 (1977).


\bibitem{Tou09} H. Touchette, The large deviation approach to statistical mechanics, {\it Phys. Rep.} {\bf 478} (2009), 1--69.% (2009).

    \bibitem{Poe27}
B.~van~der~Pol, \"Uber relaxations schwingungen, {\it Jahrb. Drahtl. Telegr. Teleph.}  {\bf 28} (1927), 178--184.

\bibitem{XY15}
F.
Xi,
G.
Yin,
Stochastic Li\'enard equations with state-dependent switching, {\it Acta Math. Appl. Sin. Engl. Ser.} {\bf 35} (2015),  893--908.

\bibitem{XY22}
L
Xie
and
L.
Yang, The Smoluchowski-Kramers limits of stochastic differential equations with irregular coefficients, {\it Stochastic Process. Appl.} {\bf 150} (2022), 91--115.

\bibitem{Xu22}
%Jie
J.
Xu, An averaging principle for slow-fast fractional stochastic parabolic equations on unbounded domains, {\it Stochastic Process. Appl.} {\bf 150} (2022), 358--396.

\bibitem{XW23}
J.
Xu,
Q.
Lian
and
J.
Wu, A strong averaging principle rate for two-time-scale coupled forward-backward stochastic differential equations driven by fractional Brownian motion, {\it Appl. Math. Optim.} {\bf 88} (2023), no. 2, Paper No. 32, 35 pp.

 \bibitem{Ver99} A. Yu. Veretennikov,  On large deviations in the averaging principle for SDEs with a full
dependence, {\it Ann. Probab.} {\bf 27} (1999), 284--296. %1999. ISSN 0091-1798. MR-1681106

\bibitem{Ver00} A. Yu. Veretennikov, On large deviations for SDEs with small diffusion and averaging,
 {\it Stochastic Process. Appl.} {\bf 89} (2000), 69--79.

\bibitem{XZW21} F.
Xi, C. Zhu, and F. Wu, On strong Feller property, exponential ergodicity and large deviations principle for stochastic damping Hamiltonian systems with state-dependent switching,
{\it J. Differential Eqs.} {\bf 286}  (2021), 856--891.


\end{thebibliography}
\end{document}